\documentclass{amsart}
\usepackage{euscript}           
\usepackage{amsmath,amsthm}     
\usepackage{amssymb}            
\usepackage{graphicx}

\DeclareMathOperator{\Map}{Map}
\newlength{\labwidth}
\newcommand{\labarrow}[1]{
\settowidth{\labwidth}{$\scriptstyle \;\; #1 \;\;$}
\stackrel{#1}{\smash{\hbox to \labwidth{\rightarrowfill}}
\vphantom{\longrightarrow}}
}
\newcommand{\bfn}{{\mathbf n}}
\newcommand{\bu}{\bullet}

\newcommand{\bQ}{{\mathbf Q}}

\newcommand{\bS}{{\mathbf S}}
\newcommand{\bM}{{\mathbf M}}

\newcommand{\st}{{\mathrm{st}}}
\newcommand{\lk}{{\mathrm{lk}}}
\newcommand{\bX}{{\mathbf X}}
\newcommand{\bY}{{\mathbf Y}}

\renewcommand{\S}{{\mathcal S}}
\renewcommand{\H}{{\mathcal H}}

\newcommand{\A}{{\mathcal A}}
\newcommand{\M}{{\mathcal M}}
\newcommand{\Cell}{{\mathcal C}ell}
\newcommand{\cl}{\mathrm{cl}}
\newcommand{\id}{{\mathrm{id}}}

\newcommand{\D}{{\mathcal D}}
\newcommand{\Set}{{\mathrm{Set}}}
\newcommand{\Ab}{{\mathrm{Ab}}}

\newcommand{\ad}{{\mathrm{ad}}}
\newcommand{\pre}{{\mathrm{pre}}}
\newcommand{\inj}{{\mathrm{inj}}}
\newcommand{\Top}{{\mathrm{Top}}}
\newcommand{\Sig}{{\mathrm{Sig}}}
\newcommand{\STop}{{\mathrm{STop}}}
\newcommand{\STopFun}{{\mathrm{STopFun}}}

\usepackage {epsf}
\DeclareMathOperator*{\colim}{colim}

\newcommand {\R}{{\mathcal R}}

\newcommand {\T}{{\mathcal T}}
\newcommand{\Msym} {{\bM}_{\mathrm{sym}}}
\newcommand{\Mquad} {{\bM}_{\mathrm{quad}}}
\newcommand{\Mgeom} {{\bM}_{\mathrm{geom}}}
\newcommand{\N} {{\mathcal N}}

\newcommand {\Z}{{\mathbb Z}}

\newcommand {\C}{{\cat C}}
\newcommand {\B}{{\cat B}}
\newcommand {\fA}{{\mathfrak A}}

\newcommand{\lra}{\longrightarrow}              

\newcommand{\holim}{\mbox{\rm ho}\! \lim }

\newcommand{\cat}{\mathcal}    

\newfont{\german}       {eufm10 at 12pt}
\DeclareMathOperator{\Hom}{Hom} 
\DeclareMathOperator{\Nat}{Nat} 

\numberwithin{equation}{section}
\newtheorem{thm}{Theorem}[section]
\newcounter{numerierer}
\newcounter{leer}

\newtheorem{defn}[thm]{Definition}
\newtheorem{prop}[thm]{Proposition}

\newtheorem{cor}[thm]{Corollary}
\newtheorem{lemma}[thm]{Lemma}

\theoremstyle{definition}  
\newenvironment{definition}{\begin{defn}\rm}{\end{defn}}

\newtheorem{example}[thm]{Example}
\newtheorem{set theory}[thm]{Set Theoretic Prelude}
\newtheorem{redefinition}[thm]{Redefinition}
\newtheorem{convention}[thm]{Convention}

\newtheorem{notation}[thm]{Notation}
\newtheorem{assumption}[thm]{Assumption}

\newtheorem{remark}[thm]{Remark}
\usepackage[frame,matrix,curve, arrow]{xy}

\xymatrixcolsep{1.9pc}                          
\xymatrixrowsep{1.9pc}
\newdir{ >}{{}*!/-5pt/\dir{>}}                  
\newdir{ |}{{}*!/-5pt/\dir{|}}                  

\begin{document}

\title{Multiplicative properties of Quinn spectra}
\author{Gerd Laures and James E.\ McClure}

\address{ Fakult\"at f\"ur Mathematik,  Ruhr-Universit\"at Bochum, NA1/66,
  D-44780 Bochum, Germany\newline\indent
Department of Mathematics, Purdue University, 150 N.\ University
Street, West Lafayette, IN 47907-2067}

\thanks{The second author was partially supported by NSF grants. He
thanks the Lord for making his work possible.}


\date{July 9, 2009}

\begin{abstract}
We give a simple sufficient condition for Quinn's ``bordism-type spectra'' to
be weakly equivalent to strictly associative ring spectra. We also show that
Poincar\'e bordism and symmetric L-theory are naturally weakly equivalent to
monoidal functors.  Part of the proof of these statements involves showing
that Quinn's functor from bordism-type theories to spectra lifts to the
category of symmetric spectra.  We also give a new account of the foundations.
\end{abstract}
\maketitle

\section{Introduction}

Our main goal in this paper and its sequel is to give a systematic account of
multiplicative properties of Quinn's ``bordism-type spectra.''   The present
paper deals with associativity and the sequel with commutativity.  

We also give a new account of the foundations, and we have made our paper
mostly self-contained in the hope that it can serve as an introduction to
\cite{MR1211640}, \cite{MR1026874}, \cite{WWIII} and other work in this area.

\subsection{Quinn's bordism-type spectra}

The Sullivan-Wall manifold structure sequence is one of the central results of
surgery theory.  In his thesis (\cite{Quinn}, also see \cite{MR0282375} and
\cite{MR0668807}) Frank Quinn showed how to interpret the Sullivan-Wall
sequence as part of the 
long exact homotopy sequence of a fiber sequence of spectra.   In particular,
for each group $G$ he constructed a spectrum ${\mathbf L}(G)$ (which is now 
called the quadratic L-spectrum of $G$) whose homotopy groups are Wall's 
groups $L_*(G)$.

The construction of ${\mathbf L}(G)$ is a special case of Quinn's general 
machine for constructing spectra from 
``bordism-type theories'' (see \cite{MR1388303}).  One can see the basic idea
of this machine by considering the example of topological bordism.%
\footnote{The smooth case is technically 
more difficult because it requires careful attention to manifolds with 
corners; the second author plans to pursue this in a future paper.} 
Let $T(\Top_k)$ denote the 
Thom space of the universal
${\mathbb R}^k$-bundle with structure group $\Top_k$.%
The usual simplicial model for this space, denoted
$S_\bu T(\Top_k)$, has as $n$-simplices the continuous maps 
\[
f:\Delta^n \to T(\Top_k).
\]
Let us consider the subobject $S^{\mathord{\pitchfork}}_\bu T(\Top_k)$ 
consisting of maps 
whose restrictions to each face of $\Delta^n$ are transverse to the zero
section $B(\Top_k)\subset T(\Top_k)$.  Note that $S^{\mathord{\pitchfork}}_\bu
T(\Top_k)$ is closed under face maps but not under degeneracy maps; that is, 
it is a {\it semisimplicial set}.%
\footnote{In the literature these are often called 
$\Delta$-sets, but that terminology seems infelicitous since the category 
that governs simplicial sets is called $\Delta$. Our terminology follows
\cite[Definition 8.1.9]{WEIB}.}  There is a concept of
homotopy in the category of semisimplicial sets (\cite[Section 6]{MR0300281}),
and a transversality argument using \cite[Section 9.6]{MR1201584}
shows that
$S^{\mathord{\pitchfork}}_\bu T(\Top_k)$ is a deformation retract of $S_\bu 
T(\Top_k)$. 

Next observe that for each simplex $f:\Delta^n\to T(\Top_k)$ in
$S^{\mathord{\pitchfork}}_\bu T(\Top_k)$, the intersections of 
$f^{-1}(B(\Top_k))$ 
with the faces of $\Delta^n$ form a manifold $\Delta^n$-ad;%
\footnote{In the literature these are often called $(n+2)$-ads.}
that is, a collection of
topological manifolds $X_\sigma$, indexed by the faces of $\Delta^n$, with 
monomorphisms $X_\tau\hookrightarrow \partial X_\sigma$ for $\tau\subsetneq
\sigma$ such that
\[
\partial X_\sigma=\colim_{\tau\subsetneq \sigma} X_\tau,
\]
where the colimit is taken in the category of
topological spaces (the simplest example of a manifold $\Delta^n$-ad is the 
collection of faces of $\Delta^n$ itself).  The $\Delta^n$-ads obtained in this way 
are of degree $k$ (that is, $\dim X_\sigma=\dim\sigma-k$).

Quinn observed that something interesting happens if one considers the
semisimplicial set of {\it all} manifold $\Delta^n$-ads of degree $k$; we 
denote this semisimplicial set by $P_k$ and its realization by $Q_k$.  It 
turns out that each $P_k$ is a Kan complex whose homotopy groups are the 
topological bordism groups (shifted in dimension by $k$) and that there are 
suspension maps $\Sigma Q_k\to Q_{k+1}$ which make the sequence $\bQ=\{Q_k\}$ 
an $\Omega$ spectrum (we give proofs of these statements in Section 
\ref{spectrum}).  In Appendix \ref{a1} we show that $\bQ$ is weakly 
equivalent to $M\Top$.

An important advantage of the construction just given is that it depends only 
on the {\it category} of topological manifolds, not on the bundle theory.  
Quinn gave an axiomatization of the structures to which one can apply this
construction, which he called bordism-type theories 
\cite[Section 3.2]{MR1388303}.  One example of a bordism-type theory arises 
from Poincar\'e $\Delta^n$-ads; in this situation transversality does not hold but 
one obtains a bordism spectrum from Quinn's construction (cf.\ Section 
\ref{geom} below).  Other important examples are Ranicki's quadratic and 
symmetric algebraic Poincar\'e $\Delta^n$-ads, which lead to a purely algebraic 
description of quadratic and symmetric L-spectra (\cite{MR1211640}; also 
see Sections \ref{sym} and \ref{quad} below).

\subsection{Previous work on multiplicative structures}

In \cite{MR560997} and \cite{MR566491}, Ranicki used product structures
on the L-groups to give product formulas for the surgery obstruction and 
the symmetric signature.  In \cite[Appendix B]{MR1211640} he observed that
these products come from pairings (in the sense of \cite{MR0137117}) at the 
spectrum level, and he used one of these pairings to give a new construction 
of the assembly map in quadratic L-theory.  He also suggested that the pairings
could be obtained from a bisemisimplicial construction.  This idea, which was
developed further in \cite{MR1694381}, is a key ingredient in our work.

\subsection{Smash products in the category of spectra}
Given spectra $E$, $F$ and $G$, a pairing in the sense of \cite{MR0137117} is a 
family of maps 
\[
E_i\wedge F_j \to G_{i+j}
\]
satisfying certain conditions.  That is, a pairing relates the spaces of the
spectra rather than the spectra themselves.  Starting in the early 1960's 
topologists realized that the kind of information given by pairings of 
spectra could be captured more effectively by using smash products of 
spectra.  The earliest constructions were in the stable
category (that is, the homotopy category of spectra). A smash product that was
defined at the spectrum level and not just up to homotopy was given in 
\cite{MR0866482}; however, this satisfied associativity and commutativity 
only up to higher homotopies, which was a source of considerable 
inconvenience.   In the early 1990's there were two independent constructions 
of categories of spectra in which the smash product was associative and 
commutative up to coherent natural isomorphism.  These were the categories 
of symmetric spectra (eventually published as \cite{MR1695653}) and the 
category of $S$-modules \cite{MR1417719}.  In these categories it is possible
to speak of strictly associative and commutative ring spectra (these are
equivalent to the $A_\infty$ and $E_\infty$ ring spectra of \cite{MR0494077}). 

A later paper \cite{MR1806878} gave a version of the category of symmetric 
spectra which was based on topological spaces rather than simplicial sets, and
this is the version that we will use. (Our reason for using symmetric spectra
rather than $S$-modules is that the former have a combinatorial flavor that 
makes them well-suited to constructions using $\Delta^n$-ads.)

\subsection{Our work}
Our goal is to relate Quinn's theory of bordism-type spectra to the theory of
symmetric spectra. As far as we can tell, Quinn's original axioms are not 
strong enough to do this.  We give a stronger set of axioms for a structure 
that we call an {\it ad theory} and we show that our axioms are satisfied by 
all of the standard examples.

Next we show that there is a functor from ad theories to symmetric spectra 
which is weakly equivalent to Quinn's spectrum construction.  We also give a 
sufficient condition (analogous to the existence of Cartesian products in the 
category of topological manifolds) for the symmetric spectrum arising from an 
ad theory to be a strictly associative ring spectrum.  Finally, we show that 
Poincar\'e bordism is naturally weakly equivalent to a monoidal (that is, 
coherently multiplicative) functor from a category $\T$ (Definition 
\ref{Nov7.2}) to symmetric spectra and that symmetric L-theory is naturally 
weakly equivalent to a monoidal functor from the category of rings with 
involution to symmetric spectra.

In the sequel we will give a sufficient condition for the symmetric spectrum
arising from an ad theory to be a strictly commutative ring spectrum.  We
will also show that Poincar\'e bordism and symmetric L-theory are naturally 
weakly equivalent to symmetric monoidal functors.  Finally, we will show that
the symmetric signature from Poincar\'e bordism to symmetric L-theory can be 
realized as a monoidal natural transformation.   This will show in particular 
that the Sullivan-Ranicki orientation $M\STop\to {\mathbb L}^\bu({\mathbb Z})$ 
is an $E_\infty$ ring map.

\subsection{Outline of the paper}

A $\Delta^n$-ad is indexed by the faces of $\Delta^n$.  We will also make use
of $K$-ads, indexed by the cells of a ball complex $K$ (i.e., a regular CW 
complex with a compatible PL structure).  In Section \ref{b} we collect some 
terminology about ball complexes from \cite[pages 4--5]{MR0413113}.

In Section \ref{ax} we give the axioms for an ad theory, together with a simple
example (the cellular cocyles on a ball complex). 

In Section \ref{groups} we define the bordism sets of an ad theory and show
that they are abelian groups.

In Sections \ref{balanced}--\ref{glue} we consider the standard examples of
bordism-type theories and show that they are ad theories; this does not follow
from the existing literature because our axioms for an ad theory (especially
the gluing axiom) are much
stronger than Quinn's axioms for a bordism-type theory.  We give careful
discussions of set-theoretic issues, which does not seem to have been done
previously.
Section
\ref{balanced} gives some preliminary terminology.  Oriented topological
bordism is treated in Section \ref{classical}, geometric Poincar\'e bordism 
in Sections \ref{geom} and \ref{moregeom}, symmetric and quadratic 
Poincar\'e bordism (using ideas from \cite{MR1026874})  in Sections \ref{sym} 
and \ref{quad}.  In Section \ref{ss} we construct the symmetric
signature as a morphism of ad theories from geometric Poincar\'e bordism to 
symmetic Poincar\'e bordism.  Section \ref{glue} gives a gluing result which 
is needed for Sections \ref{geom}--\ref{quad} and may be of independent 
interest.

It is our hope that new families of ad theories will be discovered (your ad
here).

In Section \ref{fun} we use an idea of Blumberg and Mandell to show that the 
various kinds of Poincar\'e bordism are functorial---this question seems not to
have been considered in the literature.  

In Sections \ref{cohomology section}--\ref{f11} we consider the cohomology
theory associated to an ad theory; this is needed in later sections and is
important in its own right.  There is a functor (which we denote by $T^*$)
that takes a ball complex $K$ to the graded abelian group of $K$-ads modulo 
a certain natural bordism relation.  Ranicki \cite[Proposition 
13.7]{MR1211640} stated that (for symmetric and quadratic Poincar\'e bordism, 
and assuming that $K$ is a simplicial complex) $T^*$ is the cohomology theory 
represented by the Quinn spectrum $\bQ$ (Quinn stated a similar result
\cite[Section 4.7]{MR1388303} but seems to have had a different equivalence 
relation in mind). The proof of this fact in \cite{MR1211640} is not correct
(see Remark \ref{n4} below).  We give a different proof (for general ad
theories, and general $K$).  First, in Section \ref{cohomology section} we use
ideas from \cite{MR0413113} to show that $T^*$ is a cohomology theory.
In Section \ref{spectrum} we review the construction of the Quinn spectrum
$\bQ$.  Then in Section \ref{f11} we show that $T^*$ is naturally isomorphic to
the cohomology theory represented by $\bQ$ by giving a morphism of cohomology
theories which is an isomorphism on coefficients.

In Section \ref{f12} we review the definition of symmetric spectrum and show 
that the functor $\bQ$ from ad theories to spectra lifts (up to weak 
equivalence) to a functor $\bM$ from ad theories to symmetric
spectra.  In Section \ref{mult} we consider multiplicative ad theories and show
that for such a theory the symmetric spectrum $\bM$ is a strictly associative
ring spectrum.  In Section \ref{Nov8} we show that the functors $\bM$
given by the geometric and symmetric Poincar\'e bordism ad theories are
monoidal functors.

In an appendix we review some simple facts from PL topology that are needed in
the body of the paper.

\subsection*{Acknowledgments} The authors benefited from a workshop on
forms of homotopy theory held at the Fields Institute. They would like to 
thank Matthias Kreck for suggesting the problem to the first author and also
Carl-Friedrich B\"odigheimer, 
Jim Davis,
Steve Ferry,
Mike Mandell,
Frank Quinn, 
Andrew Ranicki, 
John Rognes, 
Stefan Schwede, 
Michael Weiss 
and 
Bruce Williams
for useful hints and helpful discussions. The first 
author is  grateful to the Max Planck Institute in Bonn for its hospitality.

\section{Ball complexes}
\label{b}

\begin{definition} 
\label{ball}
(i) 
Let $K$ be a finite collection of PL balls in some ${\mathbb R}^n$, and write
$|K|$ for the union $\cup_{\sigma\in K}\, \sigma$.
We say that $K$ is a  
{\it ball complex} if the interiors of the balls of $K$ are disjoint and the
boundary of each ball of $K$ is a union of balls of $K$ (thus the interiors of
the balls of $K$ give $|K|$ the structure of a regular CW complex). The balls
of $K$ will also be called {\it closed cells} of $K$.

(ii) An {\it isomorphism} from a ball complex $K$ to a ball complex $L$ is a PL
homeomorphism $|K|\to |L|$ which takes closed cells of $K$ to closed cells of 
$L$.  

(iii)
A {\it subcomplex} of a ball complex $K$ is a subset of $K$ which is a ball
complex.

(iv)
A {\it morphism} of ball complexes is the composite of an
isomorphism with an inclusion of a subcomplex.
\end{definition}

\begin{definition}
A {\it subdivision} of a ball complex $K$ is a ball complex $K'$ 
with two properties:

(a) $|K'|=|K|$, and 

(b) each closed cell of $K'$ is contained in
a closed cell of $K$. 

A subcomplex of $K$ which is also a subcomplex of $K'$ is called {\it residual}.
\end{definition}

\begin{notation}
Let $I$ denote the unit interval with its standard structure as a ball complex 
(two 0 cells and one 1 cell).
\end{notation}

\section{Axioms}
\label{ax}

\begin{definition}
A category with involution is a category together with an
endofunctor $i$ which satisfies $i^2 =1$. 
\end{definition}

\begin{example}
The set of integers $\Z$ is a poset and therefore a category.  We give it the
trivial involution.
\end{example}

\begin{definition}
\label{j1}
A {\em $\Z$-graded category} is a small
category $\A$ with involution together with involution-preserving functors 
$ d: \A \lra \Z$ (called the {\it dimension function}) and $\emptyset : \Z 
\lra \A$
such that

a) $d \, \emptyset$ is equal to the identity functor, and

b) if $f:a\to b$ is a non-identity morphism in $\A$ then $d(a)<d(b)$.

A {\it $k$-morphism} between $\Z$-graded 
categories is a functor which decreases the dimensions of objects by $k$ and 
strictly commutes with $\emptyset$ and $i$. 
\end{definition}

We will write $\emptyset_n$ for $\emptyset(n)$.

\begin{example}\label{example 1}
Given a chain complex $C$, let
$\A_C$ be the $\Z$-graded category whose objects in dimension $n$ are the
elements of $C_n$.  There is a unique morphism $a\to b$ whenever $\dim a< \dim
b$; these are the only non-identity morphisms.
$i$ is multiplication by $-1$ and the object $\emptyset_n$ is
the 0 element in $C_n$. 
\end{example}

\begin{example}
\label{STop}
Let $\A_{\mathrm{STop}}$ be the category defined as follows.  The objects of 
dimension $n$ are the $n$-dimensional oriented compact topological manifolds 
with boundary (with an empty
manifold of dimension $n$ for each $n$); in order to ensure that $\A_\STop$ is
a small category we assume in addition that each object of $\A_\STop$ is a
subspace of some ${\mathbb R}^m$.
The non-identity morphisms are the continuous monomorphisms 
$\iota:M\rightarrow N$ with $\dim M<\dim N$ and $\iota(M)\subset \partial 
N$.  The involution $i$ reverses the orientation,  and 
$\emptyset_n$ is the empty manifold of dimension $n$. 
\end{example}

For examples related to geometric and algebraic Poincar\'e bordism see
Definitions \ref{Nov7}, \ref{quasi} and \ref{q} below.

\begin{example}
Let $K$ be a ball complex and $L$ a subcomplex.
Define $\Cell(K,L)$ to be the $\Z$-graded category whose 
objects in dimension $n$ are the oriented closed $n$-cells $(\sigma,o)$ which
are not in $L$, 
together with an object $\emptyset_n$ (the empty cell of dimension $n$).
There is a unique morphism $(\sigma,o)\to (\sigma',o')$ whenever 
$\sigma\subsetneq\sigma'$ (with no requirements on the orientations)
and a unique morphism $\emptyset_n\to(\sigma,o)$ whenever $n< \dim \sigma$;
these are the only non-identity morphisms.
The involution $i$ reverses the orientation.
\end{example}

We will write $\Cell(K)$ instead of $\Cell(K,\emptyset)$.

It will be important for us to consider abstract $k$-morphisms between
categories of the form $\Cell(K_1,L_1)$, $\Cell(K_2,L_2)$ (which will not be
induced by maps of pairs in general).
The motivation for the first part of the following definition is the fact 
that, if $f$ is a chain map which lowers degrees by $k$, then $f\circ 
\partial=(-1)^k\partial \circ f$.

\begin{definition}
\label{n25}
Let $\theta:\Cell(K_1,L_1)\to\Cell(K_2,L_2)$ be a $k$-morphism.

(i)
$\theta$ is {\it incidence-compatible} if it takes incidence numbers in 
$\Cell(K_1,L_1)$ (see \cite[page 82]{MR516508}) to $(-1)^k$ times the 
corresponding incidence numbers in $\Cell(K_2,L_2)$.

(ii)
If $\A$ is a $\Z$-graded category and
$F:\Cell(K_2,L_2)\to \A$ is an $l$-morphism define an $(l+k)$-morphism
\[
\theta^*F:\Cell(K_1,L_1)\to \A
\]
to be the composite $i^{kl}\circ F\circ\theta$.
\end{definition}

Now we fix a $\Z$-graded category $\A$.

\begin{definition}
Let $K$ be a ball complex and $L$ a subcomplex.

(i) A {\it pre $K$-ad} of degree $k$ is a $k$-morphism $\Cell(K)\to \A$.

(ii) The {\it trivial pre $K$-ad} of degree $k$ is the composite
\[
\Cell(K) \xrightarrow{d}
\Z
\xrightarrow{-k}
\Z
\xrightarrow{\emptyset}
\A.
\]

(iii) A {\it pre $(K,L)$-ad} of degree $k$ is a pre $K$-ad of degree $k$ which 
restricts to the trivial pre $L$-ad of degree $k$.
\end{definition}


We write $\pre^k(K)$ for the set of pre $K$-ads of degree $k$ and 
$\pre^k(K,L)$
for the set of pre $(K,L)$-ads of degree $k$.  

There is a canonical bijection between $\pre^k(K,L)$ and the set 
of $k$-morphisms $\Cell(K,L)\to 
\A$: given a $k$-morphism $F$ the corresponding pre $(K,L)$-ad
is $\zeta^*F$, where 
$\zeta:\Cell(K)\to \Cell(K,L)$ is defined by 
\[
\zeta(\sigma,o)=
\begin{cases}
\emptyset & \text{if $\sigma$ is in $L$}, \\
(\sigma,o) & \text{otherwise}.
\end{cases}
\]
Using this bijection and Definition \ref{n25}(ii), we see that each
$k$-morphism 
\[
\theta:\Cell(K_1,L_1)\to\Cell(K_2,L_2)
\]
determines a map
\[
\theta^*:\pre^l(K_2,L_2)\to \pre^{l+k}(K_1,L_1)
\]
for every $l$.

\begin{remark}
We could have {\it defined} $\pre^k(K,L)$ to be the set of $k$-morphisms
$\Cell(K,L)\to \A$, but for our later work it's more convenient for 
$\pre^k(K,L)$ to be a subset of $\pre^k(K)$.
\end{remark}

\begin{definition}
\label{defad}
An {\it ad theory} consists of

\smallskip

(i) a $\Z$-graded category $\A$, and

(ii) for each $k$, and each ball complex pair $(K,L)$, a subset $\ad^k(K,L)$ 
of $\pre^k(K,L)$ (called the set of {\it $(K,L)$-ads of degree $k$})

\smallskip

\noindent
such that the following hold.

\smallskip

(a) $\ad^k$ is a subfunctor of $\pre^k$, and an element of $\pre^k(K,L)$ is in
$\ad^k(K,L)$ if and only if it is in $\ad^k(K)$.

(b) The trivial pre $K$-ad of degree $k$ is in $\ad^k(K)$.

(c) $i$ takes $K$-ads to $K$-ads.

(d) A pre $K$-ad is a $K$-ad if it restricts to a $\sigma$-ad for each
closed cell $\sigma$ of $K$.

(e) (Reindexing.) Suppose
\[
\theta: \Cell(K_1,L_1)\to \Cell(K_2,L_2)
\]
is an incidence-compatible $k$-isomorphism of $\Z$-graded categories.
Then the induced bijection
\[
\theta^*:\pre^l(K_2,L_2)\to
\pre^{l+k}(K_1,L_1)
\]
restricts to a bijection
\[
\theta^*:\ad^l(K_1,L_1)\to
\ad^{l+k}(K,L).
\]

(f) (Gluing.) For each subdivision $K'$ of $K$ and each $K'$-ad $F$ there is a
$K$-ad which agrees with $F$ on each residual subcomplex.

(g) (Cylinder.) There is a natural transformation 
$J:\ad^k(K)\to\ad^k(K\times I)$ (where $K\times I$ has its canonical ball
complex structure \cite[page 5]{MR0413113}) with the following properties.

\begin{itemize}

\item
$J$ takes trivial ads to trivial ads.

\item
The restriction of $J(F)$ to $K\times 0$ is the composite
\[
\Cell(K\times 0)\cong \Cell(K)\xrightarrow{F} \A.
\]

\item
The restriction of $J(F)$ to $K\times 1$ is the composite
\[
\Cell(K\times 1)\cong \Cell(K)\xrightarrow{F} \A.
\]
\end{itemize}

We call $\A$ the {\it target category} of the ad theory.
A {\it morphism} of ad theories is a functor of target categories which takes 
ads to ads.  
\end{definition}

\begin{remark}
This definition is based in part on \cite[Section 3.2]{MR1388303} and 
\cite[Theorem I.7.2]{MR0413113}.  
\end{remark}

\begin{example}\label{gamma}
Let $C$ be a chain complex and let $\A_C$ be the $\Z$-graded category of
Example \ref{example 1}.  We define an ad-theory (denoted by $\ad_C$) 
as follows.  Let 
$\cl(K)$ 
denote the cellular chain complex of $K$; specifically, 
$\cl_n(K)$ is generated by the symbols $\langle \sigma,o\rangle$ 
with $\sigma$ 
$n$-dimensional, subject to the relation 
$\langle \sigma,-o\rangle=-\langle \sigma,o\rangle$; the boundary map is given
in the usual way by incidence numbers. A 
pre $K$-ad $F$ gives a map of graded abelian groups from
$ \cl(K)$ to $C$, and $F$ is a $K$-ad if this is a chain map.
Gluing is addition and $J(F)$ is 0 on all the objects of $K\times I$ which 
are not contained in $K\times 0$ or $K\times 1$.
\end{example}

\section{The bordism groups of an ad theory}
\label{groups}

Fix an ad theory.  Let $*$ denote the one-point space.

\begin{definition}
Two elements of $\ad^k(*)$ are {\it bordant} if there is an $I$-ad which 
restricts to the given ads at the ends.
\end{definition}

This is an equivalence relation: reflexivity follows from part (g) of
Definition \ref{defad}, symmetry from part (e), and transitivity from part
(f).  

\begin{definition}
Let $\Omega_k$ be the set of bordism  classes in $\ad^{-k}(*)$.
\end{definition}

\begin{example}
\label{bordhom}
Let $C$ be a chain complex and let $\ad_C$ be the ad theory defined in Example
\ref{gamma}.  Then a $*$-ad is a cycle of $C$ and there is a bijection between
$\Omega_k$ and $H_kC$.  We will return to this example at the end of the 
section.
\end{example}

Our main goal in this section is to show that $\Omega_k$ has an abelian group 
structure (cf.\ \cite[Section 3.3]{MR1388303}).  For this we need some 
notation.

Let $M'$ be the pushout of ball complexes
\[
\xymatrix{
I
\ar[r]^-\alpha
\ar[d]_\beta
&
I\times I
\ar[d]^\gamma
\\
I\times I
\ar[r]^\delta
&
M'
}
\]
where $\alpha$ takes $t$ to $(1,t)$ and $\beta$ takes $t$ to $(0,t)$; see
Figure 1.
\begin{center}
\begin{figure}[ht]
\includegraphics[scale=0.5]{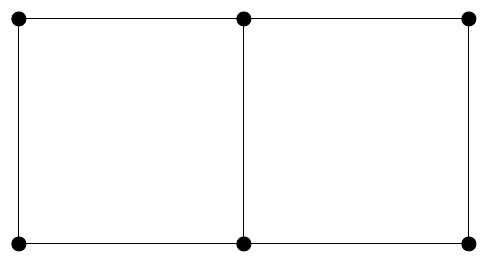}
\caption{}
\end{figure}
\end{center}
Let $M$ be the ball complex with
the same total space as $M'$ whose (closed) cells are: the union of the two 
2-cells of $M'$, the 1-cells $\gamma(I\times 0)$, 
$\delta(I\times 0)$, $\gamma(0\times I)$,
$\delta(1\times I)$ and $\gamma(I\times 1)\cup\delta(I\times 1)$, and the 
vertices of these 1-cells; see Figure 2.  
\begin{center}
\begin{figure}[ht]
\includegraphics[scale=0.5]{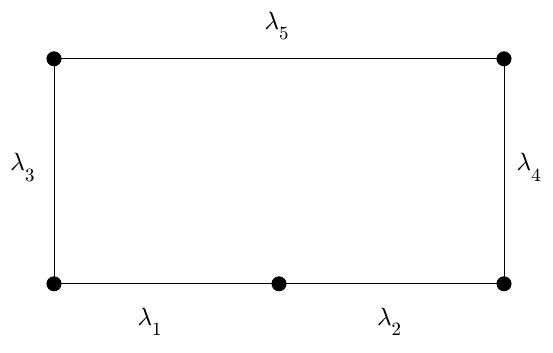}
\caption{}
\end{figure}
\end{center}
We will need explicit
parametrizations of the 1-cells of $M$: for $t\in I$ define 
\begin{align*}
\lambda_1(t) &= \gamma(t,0) \\
\lambda_2(t) &= \delta(t,0) \\
\lambda_3(t) &= \gamma(0,t) \\
\lambda_4(t) &= \delta(1,t) \\
\lambda_5(t) &= 
\begin{cases}
\gamma(2t,1) & \text{if $t\in[0,1/2]$}, \\
\delta(2t-1,1) & \text{if $t\in[1/2,1]$}
\end{cases}
\end{align*}

Let us write $\kappa$ for the isomorphism of categories
\[
\Cell(I,\{0,1\})
\to 
\Cell(*)
\]
which takes $I$ with its standard orientation to $*$ with its standard 
orientation.  The map 
\[
\kappa^*: \ad^k(*)\to\ad^{k+1}(I,\{0,1\})
\]
is a bijection by part (e) of Definition \ref{defad}.

\begin{lemma}
\label{add}
For $F,G\in \ad^k(*)$, there is an $H\in \ad^{k+1}(M)$ such that
$\lambda_1^*H=\kappa^*F$, 
$\lambda_2^*H=\kappa^*G$, and $\lambda_3^*H$ and $\lambda_4^*H$ are trivial.
\end{lemma}

\begin{proof}
By part (d) of Definition \ref{defad}, there is an $M'$-ad which restricts to
the cylinder $J(\kappa^*F)$ on the image of $\gamma$ and to the cylinder 
$J(\kappa^*G)$ on the image of $\delta$.  The result now follows by
part (f) of Definition \ref{defad}.
\end{proof}

We will write $[F]$ for the bordism class of a $*$-ad $F$.  

\begin{definition}
\label{add2}
Given $F,G\in\ad^k(*)$, let $H$ be an $M$-ad as in Lemma \ref{add} and
define $[F]+[G]$ to be 
\[
[(\kappa^{-1})^*\lambda_5^*H].
\]
\end{definition}

We need to show that this is well-defined.  Let $F_1$ and $G_1$ be bordant to
$F$ and $G$, and let $H_1$ be an $M$-ad for which 
$\lambda_1^*H_1=\kappa^*F_1$,
$\lambda_2^*H_1=\kappa^*G_1$, and $\lambda_3^*H_1$ and $\lambda_4^*H_1$ are 
trivial.  Figure 3, together with part (e) of Definition 
\ref{defad}, gives a bordism from
$[(\kappa^{-1})^*\lambda_5^*H]$ to
$[(\kappa^{-1})^*\lambda_5^*H_1]$.

\begin{center}
\begin{figure}[ht]
\includegraphics[scale=0.5]{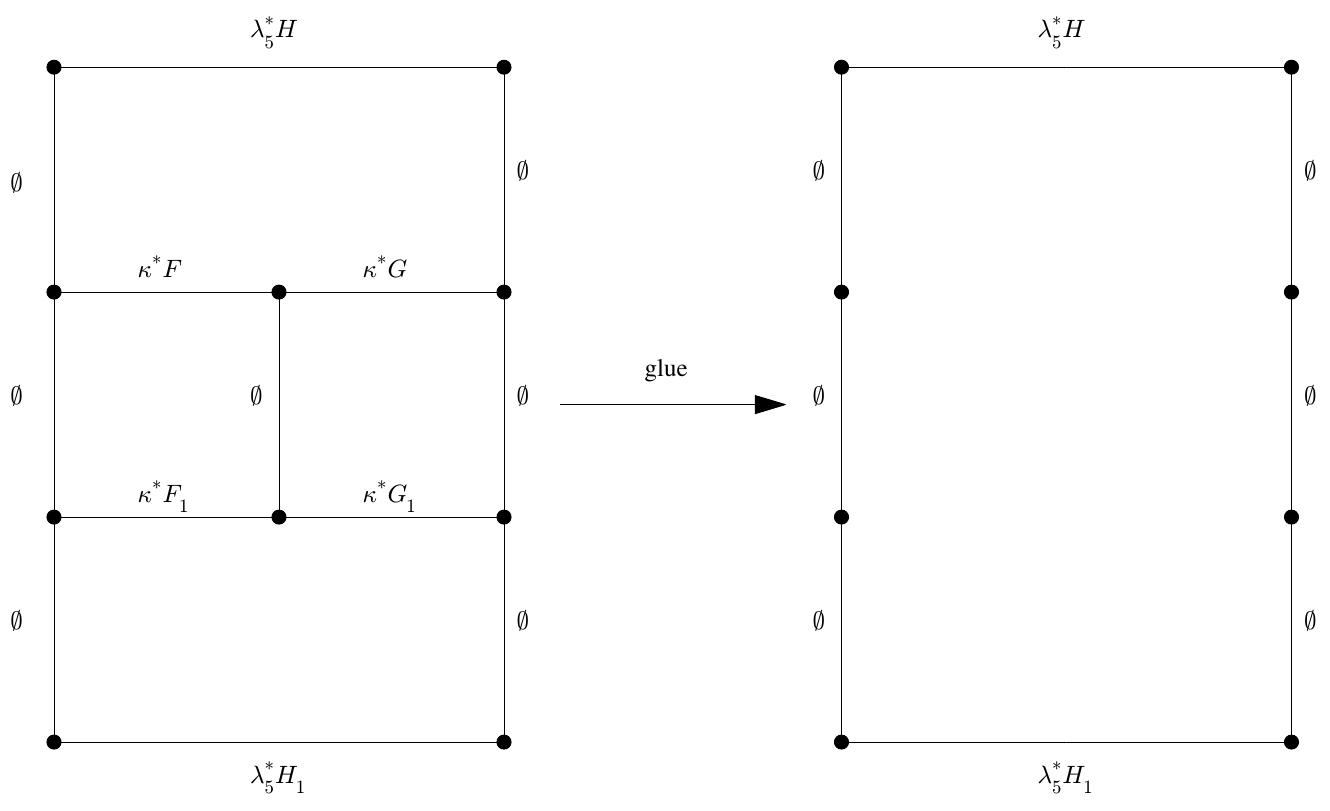}
\caption{}
\end{figure}
\end{center}

\begin{remark}
Our definition of addition agrees with that in \cite[Section 3.3]{MR1388303} 
because 
the $\Z$-graded category $\Cell(M,\lambda_3(I)\cup\lambda_4(I))$ is 
isomorphic to $\Cell(\Delta^2)$. 
\end{remark}

\begin{prop}
The operation $+$ makes $\Omega_k$ an abelian group.
\end{prop}

\begin{proof}
Let $0$ denote the bordism class of the trivial $*$-ad.
The cylinder $J(F)$, together with part (e) of Definition \ref{defad}, shows
both that $0$ is an identity element and that $[iF]$ is the inverse of $[F]$.
Figure 4, together with part (e) of Definition \ref{defad}, gives the proof of
associativity. 

\begin{center}
\begin{figure}[ht]
\includegraphics[scale=0.5]{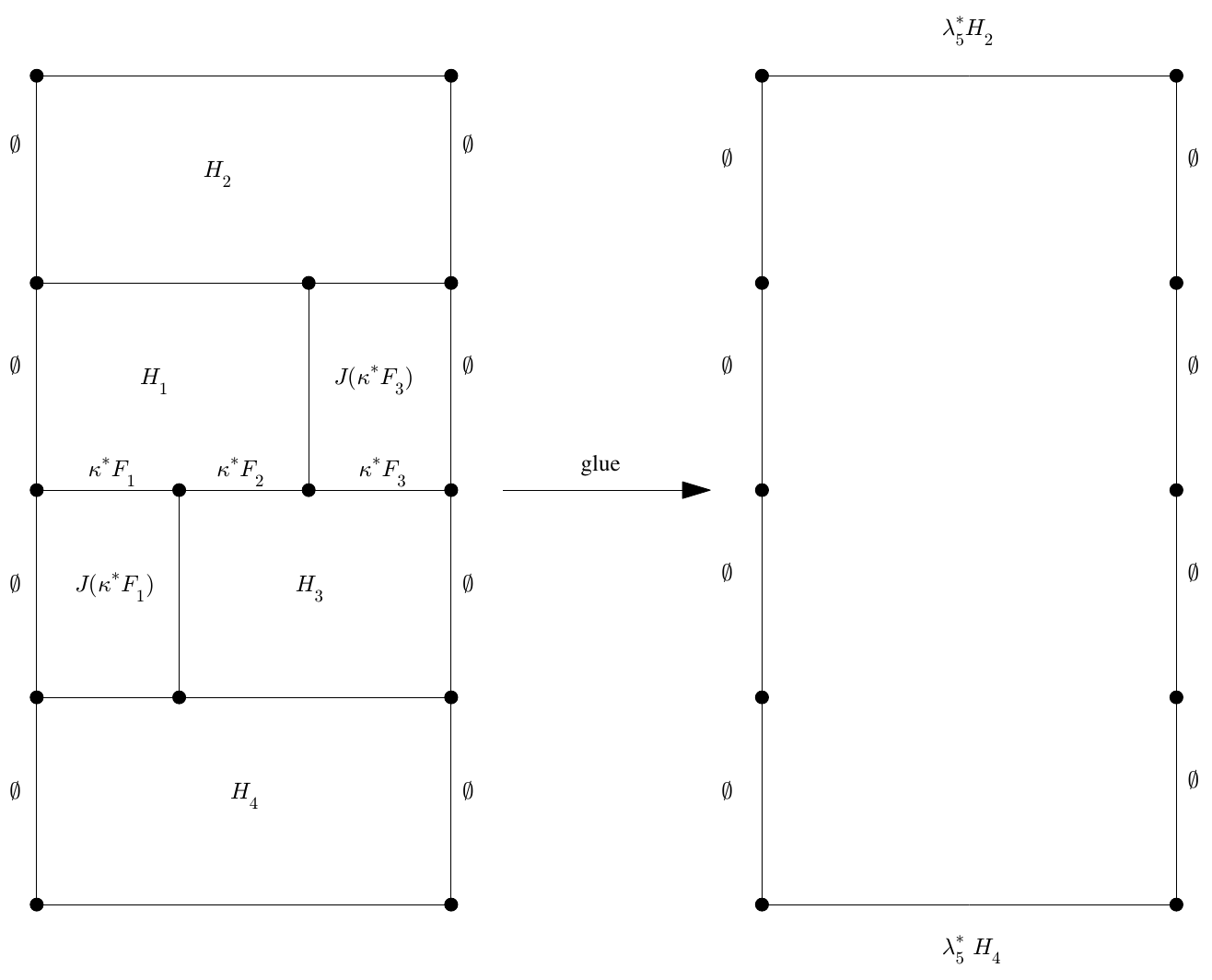}
\caption{}
\end{figure}
\end{center}

To see commutativity, let $F$, $G$ and $H$ be as in
Lemma \ref{add}.  Then $iH$ is an $M$-ad and Definition \ref{add2} gives
\[
[iF]+[iG]=[(\kappa^{-1})^*\lambda_5^* (iH)].
\]
The left-hand side of this equation is equal to $-[F]+(-[G])$, 
and the right-hand side is $-([F]+[G])$; this implies that $+$ is commutative.
\end{proof}

\begin{remark}
In Example \ref{bordhom}, the addition in $\Omega_k$ is induced by addition in
$C$, as one can see from the proof of Lemma \ref{add} and the fact that gluing
in $\ad_C$ is given by addition.
Thus $\Omega_k$ is isomorphic to $H_k C$ as an abelian group.
\end{remark}

\section{Balanced categories and functors}
\label{balanced}

For the examples in Sections \ref{classical}--\ref{quad}, it will be convenient
to have some additional terminology.

Let $\A(A,B)$ denote the set of morphisms in $\A$ from $A$ to $B$.

\begin{definition}
A {\it balanced category} is a $\Z$-graded category $\A$ together with a 
natural bijection 
\[
\eta : \A (A,B) \to \A(A,i(B))
\]
for objects $A,B$ with $\dim A< \dim B$, such that 

(a)
$\eta \circ i=i\circ\eta:\A (A,B) \to \A(i(A),B)$, and

(b)
$\eta\circ\eta$ is the identity.

If $\A$ and $\A'$ are balanced categories then
a {\it balanced functor} 
$F:\A\to \A'$ is a morphism of $\Z$-graded categories for which
\[
F\circ \eta=\eta\circ F:\A(A,B)\to \A'(F(A),i(F(B))).
\]
\end{definition}  

All of the $\Z$-graded categories in the previous section are balanced.  In
particular $\Cell(K,L)$ is balanced.

\begin{definition}
Let $\A$ be a balanced category.  A {\it balanced pre $(K,L)$-ad} with values
in $\A$ is a pre $(K,L)$-ad $F$ which is a balanced functor.
\end{definition}

\section{Example: Oriented Topological Bordism}
\label{classical}

In this section we construct an ad theory with values in the category
$\A_{\mathrm{STop}}$ of Example \ref{STop}.

Define a category $\B$ as follows:  the objects of 
$\B$ are compact orientable topological manifolds with boundary and the
non-identity morphisms are the continuous monomorphisms $\iota:M\rightarrow 
N$ with $\dim M< \dim N$ and $\iota(M)\subset \partial N$.

\begin{definition}
\label{circ}
For a ball complex $K$, let $\Cell^\flat(K)$ denote the category whose 
objects are the cells of $K$ (including an empty cell in each dimension) and 
whose morphisms are the inclusions of cells.   
\end{definition}

A balanced pre $K$-ad $F$ with values in $\A_{\mathrm{STop}}$ 
induces a functor
\[
F^\flat:\Cell^\flat(K)\to \B.
\]

Given cells $\sigma'\subsetneq\sigma$ of $K$,
let $i_{(\sigma',o'),(\sigma,o)}$ denote the map in $\Cell(K)$ from 
$(\sigma',o')$ to $(\sigma,o)$ and let $j_{\sigma',\sigma}$ denote the map in
$\Cell^\flat(K)$ from $\sigma'$ to $\sigma$.

\begin{definition}
Let $K$ be a ball complex.  A $K$-ad with values in $\A_\STop$ is a 
balanced pre $K$-ad $F$ with the following properties.

(a)
If $(\sigma',o')$ and $(\sigma,o)$ are oriented cells with
$\dim\sigma'=\dim\sigma-1$ and if
the incidence number $[o,o']$ is equal
to $(-1)^k$ (where $k$ is the degree of $F$) 
then the map
\[
F(i_{(\sigma',o'),(\sigma,o)}):F(\sigma',o')\to\partial F(\sigma,o)
\]
is orientation preserving.

(b)
For each $\sigma$, $\partial F^\flat(\sigma)$ is the colimit in Top of
$ F^\flat|_{\Cell^\flat(\partial \sigma)}.$ 
\end{definition}

\begin{remark}
\label{n28}
The sign in part (a) of this definition is needed in order for part (e) of 
Definition \ref{defad} to hold.
\end{remark}

\begin{example}
The functor $\Cell(\Delta^n)\to\A_\STop$ which takes each oriented simplex of
$\Delta^n$ to itself (considered as an oriented topological manifold) is a
$\Delta^n$-ad of degree 0.
\end{example}

We write $\ad_{\mathrm{STop}}(K)$ for the set of $K$-ads with values in 
$\A_\STop$.  

\begin{thm}
\label{Sep30}
$\ad_{\mathrm{STop}}$ is an ad theory.
\end{thm}

The rest of this section is devoted to the proof of Theorem \ref{Sep30}.
The only parts of Definition \ref{defad} which are not obvious are (f) and (g).

For part (g),
let $F$ be a $K$-ad; we need to define $J(F): \Cell(K\times I)\to
\A_\STop$.  First note that the statement of part (g) specifies what $J(F)$ has to 
be on the subcategories $\Cell(K\times 0)$ and $\Cell(K\times 1)$.  The
remaining objects have the form 
$(\sigma \times I,o\times o')$ and we define $J(F)$ for such an object to be 
$F(\sigma,o)\times (I,o')$, where $(I,o')$ denotes the topological manifold 
$I$ with orientation $o'$.  Next observe that, since $J(F)$ is to be a balanced
functor, we need to define $J(F)^\flat$ on the morphisms of
$\Cell^\flat(K\times I)$.  These morphisms are generated by those of the
following five kinds:

i) $(\sigma\hookrightarrow\sigma')\times 0$,

ii) $(\sigma\hookrightarrow\sigma')\times 1$,

iii) $(\sigma\hookrightarrow\sigma')\times I$,

iv) $\sigma\times (0\hookrightarrow I)$,

v) $\sigma\times (1\hookrightarrow I)$.

For the first two kinds of morphisms the definition of $J(F)$ is prescribed by
the statement of part (g).  For the third kind we define 
\[
J(F)((\sigma\hookrightarrow\sigma')\times I)
=
F(\sigma\hookrightarrow\sigma')\times I.
\]
For the fourth kind we define
\[
J(F)(\sigma\times (0\hookrightarrow I))
=
F(\sigma)\times (0\hookrightarrow I),
\]
and similarly for the fifth kind.

For part (f), let $K$ be a ball complex and $K'$ a subdivision of $K$.  The 
proof is by induction on the lowest dimensional cell of $K$ which is not 
a cell of $K'$.  For the inductive step, we may assume that $|K|$ is a PL
$n$-ball, that $K$ has exactly one $n$ cell, and 
that $K'$ is a subdivision of $K$ which agrees with $K$ on the boundary of 
$|K|$.  Let $F$ be a $K'$-ad.  It suffices to show that the colimit of 
$F^\flat$ over the cells of $K'$ is a topological manifold with boundary
and that its boundary  is the colimit of $ F^\flat$ over the cells of the 
boundary of $|K|$.  

We will prove something more general:

\begin{prop}
\label{f2}
Let $(L,L_0)$ be a ball complex pair such that $|L|$ is a PL manifold with 
boundary $|L_0|$.  Let $F$ be an $L$-ad.  Then 
$\colim_{\sigma\in L} F^\flat(\sigma)$ is a topological manifold with 
boundary $\colim_{\sigma\in L_0} F^\flat(\sigma)$.
\end{prop}

\begin{proof}
(The proof is essentially the same as the proof of Lemma II.1.2 in 
\cite{MR0413113}.) 

Using the notation of the Appendix,
let us write $D^\circ(\sigma)$
for $D(\sigma)-\dot{D}(\sigma)$.  
If $\sigma$ is not in $L_0$ then,  
by Proposition \ref{f5}(i),  
$D^\circ(\sigma)$ is topologically homeomorphic to ${\mathbb R}^{n-m}$.
If $\sigma$ is in $L_0$ then, 
by Proposition \ref{f5}(ii) and \cite[Theorem 3.34]{MR665919},
there is a 
homeomorphism from $D^\circ(\sigma)$ to the half space ${\mathbb R}_{\geq 
0}\times {\mathbb R}^{n-m-1}$ which takes $\hat{\sigma}$ to a point on the 
boundary.

There is another way to describe $D^\circ(\sigma)$.  Given a (possibly empty)
sequence $T=(\sigma_1,\ldots,\sigma_l)$ with 
$\sigma\subsetneq\sigma_1\subsetneq\cdots\subsetneq\sigma_l$, let us write 
$[0,1)^T$ for 
$[0,1)^l$ and $T[i]$ for the sequence obtained by deleting $\sigma_i$.  Given 
$u\in [0,1)^l$ let us write $u[i]$ for the element of $[0,1)^{l-1}$ obtained by 
deleting the $i$-th coordinate of $u$.  Let $E(\sigma)$ be the 
quotient of
\[
\coprod_T\, [0,1)^T
\]
in which a point $u$ in $[0,1)^T$ with $i$-th coordinate 0 is identified with 
the point $u[i]$ in $[0,1)^{T[i]}$.  Let $\mathbf 0$ denote the equivalence 
class of
$(0,\ldots,0)\in [0,1)^T$ (which is independent of $T$).  Then there is a
homeomorphism $D^\circ(\sigma)\to E(\sigma)$ which takes $\hat{\sigma}$ to 
$\mathbf 0$.

Now consider the space $X=\colim_{\sigma\in L} F^\flat(\sigma)$.  Let $x\in X$.
There is a unique $\sigma$ for which $x$ is in the interior of
$F^\flat(\sigma)$.  Let $m$ be the dimension of $\sigma$, and $k$ the degree of
$F$.  Let $U$ be an $(m-k)$-dimensional Euclidean neighborhood of $x$ in 
$F^\flat(\sigma)$.  An easy inductive argument, using the collaring theorem 
for topological manifolds, gives an imbedding 
\[
h:U\times E(\sigma)\to X
\]
such that $h(x,{\mathbf 0})=x$ and $h(U\times E(\sigma))$ contains a 
neighborhood of $x$ in $X$.  If $\sigma$ is not a cell of
$L_0$ this shows that $x$ has an $(n-k)$-dimensional Euclidean
neighborhood in $X$.  If $\sigma$ is a cell of $L_0$
we obtain a homeomorphism from a neighborhood of $x$ in $X$ to the 
half space of dimension $n-k$ which takes $x$ to a boundary point.
\end{proof}

\begin{remark}
The description of gluing in the proof of Theorem \ref{Sep30}, together with 
the proof of Lemma \ref{add}, shows that addition in the bordism groups of 
$\ad_{\STop}$ is induced by disjoint union.  Thus the bordism groups are the 
usual oriented topological bordism groups.
\end{remark}

\section{Example: Geometric Poincar\'e ad Theories}
\label{geom}

Fix a group $\pi$ and a properly discontinuous left action of $\pi$ on a simply
connected space $Z$; then $Z$ is a universal cover of $Z/\pi$. 

Fix a homomorphism $w:\pi\to\{\pm 1\}$. 

Ranicki \cite[page 243]{MR566491} defines the bordism groups 
$\Omega^P_*(Z/\pi,w)$ 
of geometric Poincar\'e complexes over $(Z/\pi,w)$; our goal in this section 
is to define an ad theory whose bordism groups are a slightly modified 
version of Ranicki's (see Section \ref{moregeom} for a precise comparison).

Let $\Z^w$ denote the right $\pi$ action on $\Z$ determined by $w$.    

\begin{definition}
\label{Nov7.1}
Given a map $f:X\to Z/\pi$, define $S_*(X,{\mathbb Z}^f)$ to be $\Z^w 
\otimes_{\Z[\pi]} S_*(\widetilde{X})$, where $\widetilde{X}$ is the pullback of 
$Z$ to $X$ and $S_*(\widetilde{X})$ denotes the singular chain complex of
$\widetilde{X}$.
\end{definition}

\begin{set theory}
\label{j3}
In order to ensure that the category $\A_{\pi,Z,w}$ that we are about to define
is small, as required by Definition \ref{j1}, we note that there is a set
$\mathfrak X$ with the following properties. 
\begin{itemize}
\item
The elements of $\mathfrak X$ are topological spaces.
\item
Every subspace of every ${\mathbb R}^n$ is in $\mathfrak X$.
\item
The Cartesian product of two spaces in $\mathfrak X$ is in $\mathfrak X$.
\item
$\mathfrak X$ is closed under pushouts. 
\end{itemize}
The verification that there is such an $\mathfrak X$ is left to the reader.
The reason for requiring these properties can be seen from the proofs of
Theorem \ref{Oct6} and Lemma \ref{j4}. 
\end{set theory}

\begin{definition} 
\label{Nov7}
We define a category 
$\A_{\pi,Z,w}$ as follows.  An object of $\A_{\pi,Z,w}$ is a triple
\[
(X,\,f:X\to Z/\pi,\,\xi\in S_*(X,{\mathbb Z}^f)),
\] 
where 
$X$ is a space in $\mathfrak X$ which has the homotopy type of a finite CW
complex.
Non-identity morphisms $(X,f,\xi)\to (X',f',\xi')$ exist only when 
$\dim \xi< \dim\xi'$, in which case the morphisms
are the maps $g:X\to X'$ such that 
$f'\circ g=f$.  
\end{definition}

\begin{remark}
\label{j2} 
(i) An important special case will be the category $\A_{e,*,1}$, where 
$e$ is the trivial group, $*$ is the one-point space and $1$ is the 
homomorphism from $e$ to $\{\pm 1\}$.

(ii) For technical reasons we will make a small change in the definition of
$\A_{\pi,Z,w}$ in Section \ref{ss}.
\end{remark}

$\A_{\pi,Z,w}$ is a balanced $\Z$-graded category, where the 
dimension of
$(X,f,\xi)$ is $\dim \xi$,
$i$ takes
$(X,f,\xi)$ to $(X,f,-\xi)$, and $\emptyset_n$ is the $n$-dimensional object 
with $X=\emptyset$.

Next we must say what the $K$-ads with values in 
$\A_{\pi,Z,w}$ are.  We will build this up gradually by considering 
several properties that a pre $K$-ad can have.

For a balanced pre $K$-ad $F$ we will use the notation
\[
F(\sigma,o)=(X_\sigma, f_\sigma,\xi_{\sigma,o});
\]
note that $X_\sigma$ and $f_\sigma$ do not depend on $o$. 

Recall Definition \ref{circ}.

\begin{definition} (cf.\ \cite[page 50]{MR1026874})
\label{well}
A functor $X$ from $\Cell^\flat(K)$ to topological spaces is
{\it well-behaved} if the following conditions hold:

(a)
For each inclusion $\tau\subset\sigma$, the map $X_\tau\to X_\sigma$ is a 
cofibration.

(b)
For each cell $\sigma$ of $K$, the map
\[
\colim_{\tau\subsetneq\sigma} \, X_\tau\to X_\sigma
\]
is a cofibration.
\end{definition}

If $F$ is a balanced pre $K$-ad for which $X$ is well-behaved,
let $X_{\partial \sigma}$ denote $\colim_{\tau\subsetneq\sigma} \, X_\tau$,
and let $\widetilde{X}_{\partial\sigma}$ be the pullback of $Z$ to 
$X_{\partial\sigma}$.

In order to describe
the Poincar\'e duality property that a $K$-ad should have,
we need some preliminary definitions.

We give the ring $\Z[\pi]$ the $w$-twisted involution (see 
\cite[page 196]{MR566491}).

\begin{definition}
Given a ring $R$ with involution and a left $R$-module $M$, define $M^t$ to be 
the right $R$-module obtained from the involution.
\end{definition}

\begin{definition}
Let $(\sigma,o)$ be an oriented cell of $K$.
Define $\zeta_{\sigma,o}$ be the image
of $\xi_{\sigma,o}$ under the map 
\begin{multline*}
\Z^w
\otimes_{\Z[\pi]}
S_*(\widetilde X_\sigma)
\stackrel{1\otimes AW}{\longrightarrow}
\Z^w
\otimes_{\Z[\pi]} 
(S_*(\widetilde X_\sigma)\otimes S_*(\widetilde X_\sigma))
\cong
S_*(\widetilde X_\sigma)^t \otimes_{\Z[\pi]} 
S_*(\widetilde X_\sigma)
\\
\to
S_*(\widetilde X_\sigma)/S_*(\widetilde{X}_{\partial\sigma})^t
\otimes_{\Z[\pi]}
S_*(\widetilde X_\sigma),
\end{multline*}
where $AW$ is the Alexander-Whitney map. 
\end{definition}

Our next definition gives a sufficient condition for $\zeta_{\sigma,o}$ to be a
cycle.

\begin{definition} 
\label{j5}
$F$ is {\it closed} if
for each $(\sigma,o)$ the chain $\partial \xi_{\sigma,o}$ 
is the sum of the images in $S_*(X_\sigma,\Z^{f_\sigma})$ of the chains
$\xi_{\sigma',o'}$,
where $(\sigma',o')$ runs through the oriented cells
for which the incidence number $[o,o']$ is $(-1)^{\deg F}$
(see Remark \ref{n28} for an explanation of the sign).
\end{definition}

\begin{remark}
An equivalent definition of closed uses the functor $\cl$ defined in
Example \ref{gamma}.  Given a cell $\sigma$ of $K$ there is a map of graded 
abelian groups
\[
\cl(\sigma)\to S_*(X_\sigma,\Z^{f_\sigma})
\]
which takes $\langle \tau,o\rangle$ to the image of $\xi_{\tau,o}$ in 
$S_*(X_\sigma,\Z^{f_\sigma})$.  $F$ is closed if this is a chain map for each
$\sigma$.
\end{remark}

\begin{convention}
From now on we will often use the convention that a cochain complex can be
thought of as a chain complex with the opposite grading.  For example, this is
needed in our next definition. 
\end{convention}

\begin{definition}
\label{slant}
Let $C$ and $D$ be chain complexes of left $R$-modules for some ring $R$ with
involution.
Define a chain map
\[
\Hom_R(D,R)\otimes (C^t\otimes_R D)\to C^t,
\]
called the {\it slant product},
by
\[
x\setminus (\alpha\otimes\beta) =(-1)^{|x||\alpha|} \alpha\cdot x(\beta).
\]
\end{definition}

Since $H_*(C^t)$ is the same graded abelian group as $H_*(C)$, 
the slant product induces a map
\[
H^i(\Hom_R(D,R))\otimes H_j(C^t\otimes_R D)\to H_{j-i}C
\]
for each $i,j$.

\begin{definition}
\label{geomad}
$F$ is a $K$-ad if 

(a)
it is balanced and closed and $X$ is well-behaved,  
and

(b) for each $(\sigma,o)$ the slant product with $\zeta_{\sigma,o}$ is an 
isomorphism
\[
H^*(\Hom_{\Z[\pi]} (S_*(\widetilde X_{\sigma}),{\Z[\pi]}))
\to
H_{\dim\sigma-\deg F-*}(\widetilde X_\sigma,\widetilde{X}_{\partial\sigma}).
\]
\end{definition}

We write $\ad_{\pi,Z,w}(K)$ for the set of $K$-ads with values in
$\A_{\pi,Z,w}$.

\begin{thm}
\label{Oct6}
$\ad_{\pi,Z,w}$ is an ad theory.
\end{thm}

For the proof we need a definition and a lemma.
For $i=1,2$, suppose given
a group $\pi_i$, a properly discontinuous left action of $\pi_i$ on a simply
connected space $Z_i$, and a homomorphism $w_i:\pi\to\{\pm 1\}$.

\begin{definition}
(i) For $i=1,2$, let $(X^i,f^i,\xi^i)$ be an object of 
$\A_{\pi_i,Z_i,w_i}$.  Define 
\[
(X^1,f^1,\xi^1)\times (X^2,f^2,\xi^2)
\]
to be the following object of $\A_{\pi_1\times \pi_2,Z_1\times Z_2,w_1 \cdot
w_2}$:
\[
(X^1\times X^2, f^1\times f^2, \xi^1\times \xi^2).
\]

(ii) For $i=1,2$, suppose given 
a ball complex $K_i$ and a
pre $K_i$-ad $F_i$ of degree $k_i$ with values in $\A_{\pi_i,Z_i,w_i}$. 
Define a pre $(K_1\times K_2)$-ad $F_1\times F_2$ with values in  
$\A_{\pi_1\times \pi_2,Z_1\times Z_2,w_1 \cdot w_2}$
by 
\[
(F_1\times F_2)(\sigma\times\tau,o_1\times o_2)
=
i^{k_2\dim\sigma}F_1(\sigma,o_1)\times F_2(\tau,o_2).
\]
\end{definition}

\begin{lemma}  
\label{prod}
For $i=1,2$, suppose given
a ball complex $K_i$ and a 
$K_i$-ad $F_i$ with values in $\A_{\pi_i,Z_i,w_i}$.
Then $F_1\times F_2$ is a $(K_1\times K_2)$-ad. \qed
\end{lemma}

\begin{proof}[Proof of \ref{Oct6}.]

The only parts of Definition \ref{defad} which are not obvious are (f)
and (g).

Part (f).
Let $F$ be a $K'$-ad with 
\[
F(\sigma,o)=(X_\sigma,f_\sigma,\xi_{\sigma,o}).
\]
We need to define a $K$-ad $E$ which agrees with $F$ on each
residual subcomplex of $K$.  As in the proof of Theorem \ref{Sep30}, 
we may assume by induction that $K$ is a ball complex structure for the $n$
disk with one $n$ cell $\tau$, and
that $K'$ is a subdivision of $K$ which agrees with $K$ on the boundary. 
We only need to define $E$ on the top cell $\tau$ of $K$.  We define
$E(\tau,o)$ to be $(V_\tau, e_\tau, \theta_{\tau,o})$, where
\begin{itemize}
\item
$V_\tau=\colim_{\sigma\in K'} X_\sigma$,
\item
$e_\tau:V_\tau\to Z/\pi$ is the obvious map, and
\item
$\theta_{\tau,o}$ is
\[
\sum_{(\sigma,o')} \xi_{\sigma,o'},
\]
where $(\sigma,o')$ runs through the $n$-dimensional cells of $K'$
with orientation induced by $o$.  
\end{itemize}
We need to check that $V_\tau$ has the homotopy type of a finite CW complex.  
$V_\tau$ can be built up by iterated pushouts, and because $F$ is 
well-behaved each of these is homotopy equivalent to the corresponding 
homotopy pushout.  The result now follows from the fact that a homotopy 
pushout of spaces which are homotopy equivalent to finite CW complexes
is also homotopy equivalent to a finite CW complex (which we leave as an 
exercise for the reader).

To conclude the proof of part (f) we note that
$E$ is closed by Proposition \ref{f7}(ii) and that part (b) of Definition
\ref{geomad}
follows from Proposition \ref{superglue} below.

For part (g)
we need a preliminary definition.  Recall Remark \ref{j2}.
Define an $I$-ad $G$ with values in $\A_{e,*,1}$ as follows.
For a cell $\sigma$ of $I$, the identity map id of $\sigma$ is a 
singular chain of the space $\sigma$; define $G(\sigma,o)$ to be $(\sigma,f,\pm
\mathrm{id})$, where $f$ is the map to a point and the $\pm$ is $+$ iff 
$o$ is the standard orientation of $\sigma$.

Now let $F$ be a $K$-ad.  We define $J(F)$ on objects $(\sigma\times I,o\times
o')$ to be $F(\sigma,o)\times G(I,o')$.  The rest of the definition of $J(F)$
is analogous to the corresponding part of the proof of Theorem \ref{Sep30}.  
$J(F)$ is an ad because it is isomorphic to $F\times G$.
\end{proof}

\begin{remark}
The description of gluing in the proof just given, together with the proof of
Lemma \ref{add}, shows that addition in the bordism groups of 
$\ad_{\pi,Z,w}$ is induced by disjoint union.  
\end{remark}

\section{More about Geometric Poincari\'e ad theories}
\label{moregeom}

In this section
we prove some facts about the bordism groups of $\ad_{\pi,Z,w}$, and we also
consider a relation between $\ad_\STop$ and $\ad_{e,*,1}$ (see Remark 
\ref{j2} for the notation). 

The definition of $\ad_{\pi,Z,w}$ in Section \ref{geom} depended upon the
choice of $\mathfrak X$ in Set Theoretic Prelude \ref{j3}.  Let ${\mathfrak X}'$
be a set which contains $\mathfrak X$ and satisfies the conditions in \ref{j3},
and let $\ad'_{\pi,Z,w}$ be the resulting ad theory.

\begin{lemma}
\label{j4}
The morphism $\ad_{\pi,Z,w}\to \ad'_{\pi,Z,w}$ induces an isomorphism of 
bordism groups.
\end{lemma}

\begin{proof}
To see that
the map of bordism groups is onto we note that every object of
${\mathfrak X}'$ is homotopy equivalent to an object of $\mathfrak X$ (because
every finite CW complex can be imbedded in Euclidean space) and that
the mapping cylinder of a homotopy equivalence of $*$-ads is an $I$-ad.
To see that it is one-to-one, let $F$ be an $I$-ad in ${\mathfrak X}'$ with
\[
F(\sigma,o)=(X_\sigma,f_\sigma,\xi_{\sigma,o})
\]
and suppose that $X_0$ and $X_1$ are in $\mathfrak X$.
Let $Y$ be an object of $\mathfrak X$ which is 
homotopy equivalent to $X_I$; replacing $X_I$ by the double
mapping cylinder
\[
(X_0\times I) \cup Y \cup (X_1\times I)
\] 
gives the required bordism in $\mathfrak X$. 
\end{proof}

Next we
compare the bordism groups of $\ad_{\pi,Z,w}$ with the groups 
$\Omega^P_*(Z/\pi,w)$ defined in \cite[page 243]{MR566491}.  Our definition 
differs from Ranicki's in two ways.  First of all, a $*$-ad 
in our sense is a triple $(X,f:X\to Z/\pi, \xi\in Z_n(X,\Z^f))$ but a
geometric Poincar\'e complex over $(Z/\pi,w)$ in Ranicki's sense is a triple 
$(X,f:X\to Z/\pi, [X]\in H_n(X,\Z^f))$.  This does not affect the bordism 
groups because of the following lemma.

\begin{lemma} 
\label{j6}
Let $(X,f:X\to Z/\pi, \xi)$ be a 
$*$-ad, and let $\xi'$ be a cycle homologous to $\xi$.  Then the
$*$-ads $(X,f:X\to Z/\pi, \xi)$ and $(X,f:X\to Z/\pi, \xi')$ are bordant.
\end{lemma}

\begin{proof} Since $\xi'$ is homologous to $\xi$ there is a chain $\theta$ 
with 
\[
d\theta=\xi'-\xi.
\]
Define an $I$-ad $H$ by letting $H$ take the cells 0,1 and I (with their
standard orientations) respectively to $(X,f,\xi)$, $(X,f,\xi')$, and 
$(X\times I,h,\xi\times \iota+ \theta\times \kappa)$,
where $h$ is the composite of the projection $X\times I\to X$ with $f$, $\iota$
is the chain given by the identity map of $I$, and $\kappa$ is the 0-chain 
represented by the point 1.  
\end{proof}

The second difference between our definition and Ranicki's is that in
\cite{MR566491} the symbol $\widetilde{X}$ denotes a universal cover of $X$ 
(that is, a cover which is universal on each component).  This presumably 
means that our bordism groups are different from those in \cite{MR566491}.  
Our reason for making this change is that the definition we give is somewhat 
simpler and seems to provide the natural domain for the symmetric signature 
(see Section \ref{ss}).  One could, if desired, modify our definition so that 
the bordism groups would be equal to those in \cite{MR566491}.

We conclude this section with a relation between $\ad_\STop$ and $\ad_{e,*,1}$.
Intuitively one would expect a morphism of ad 
theories $\ad_\STop\to \ad_{e,*,1}$, but an object of $\A_\STop$ does not 
determine an object of $\A_{e,*,1}$ because it doesn't come equipped with a 
chain $\xi$.  Instead we will construct a diagram
\[
\ad_\STop
\leftarrow
\ad_\STopFun
\rightarrow
\ad_{e,*,1}
\]
in which the first morphism induces an isomorphism of bordism groups (and
therefore a weak equivalence of Quinn spectra).

Let $\A_{\mathrm{STopFun}}$ be the category defined as follows.  The 
objects of dimension $n$ are pairs $(M,\xi)$, where $M$ is an $n$-dimensional 
oriented compact topological 
manifold with boundary which is a subspace of some ${\mathbb R}^n$ and $\xi\in
S_n(M)$ is a representative for the fundamental class of $M$.
The non-identity morphisms are the continuous monomorphisms
$\iota:M\rightarrow N$ with $\dim M<\dim N$ and $\iota(M)\subset \partial 
N$.  The involution $i$ reverses the orientation,  and
$\emptyset_n$ is the empty manifold of dimension $n$.

There is a forgetful functor $\A_\STopFun\to \A_\STop$,
and we define a $K$-ad with values in
$\A_\STopFun$ to be a pre $K$-ad such that
\begin{itemize}
\item
its image under the forgetful functor is a $K$-ad, and
\item
it satisfies Definition \ref{j5}.
\end{itemize}

The proofs of Theorems \ref{Sep30} and \ref{Oct6} show that this is an ad
theory; we denote it by $\overline{\ad}_\STop$.  
The proof of Lemma \ref{j6} shows that the morphism
$\overline{\ad}_\STop\to\ad_\STop$ gives an isomorphism of bordism groups as
required.  

Finally, we define a morphism $\A_\STopFun\to \A_{e,*,1}$
by taking $(M,\xi)$ to $(M,f,\xi)$, where $f$ is the map to a point (to see
that this really lands in $\A_{e,*,1}$ we use the fact, proved in
\cite{MR0645390}, that a
compact topological manifold is homotopy equivalent to a finite complex).  This
gives a morphism of ad theories 
$\A_\STopFun \rightarrow \A_{e,*,1}$ as required
(using the collaring theorem for topological manifolds to verify Definition
\ref{well}).

\section{Example: Symmetric Poincar\'e ad theories}
\label{sym}

Recall Set Theoretic Prelude \ref{j3}.

\begin{set theory}
\label{j8}
We note that there is a set $\mathfrak S$ with the following properties.
\begin{itemize}
\item
The elements of $\mathfrak S$ are sets.
\item
Every finite subset of $\mathbb Z$ is in $\mathfrak S$.
\item
For every $X\in\mathfrak X$ and every $n$ the set of continuous maps 
$\Delta^n\to X$ is in $\mathfrak S$.
\item
The Cartesian product of two sets in $\mathfrak S$ is in $\mathfrak S$.
\item
$\mathfrak S$ is closed under pushouts.
\end{itemize}
The verification that there is such an $\mathfrak S$ is left to the reader.
The reason for requiring these properties can be seen from 
the proof of Theorem \ref{Oct16} and Section \ref{ss}.
\end{set theory}

Fix a ring $R$ with involution.   

\begin{definition}
\label{j7}
(i) The {\it free $R$-module generated by a set
$A$}, denoted $R\langle A\rangle$, is the set of functions from $A$ to $R$ 
which are nonzero for only finitely many elements of $A$.  

(ii) Let $\M$ be the category of left $R$-modules of the form $R\langle
A\rangle$ with $A\in {\mathfrak S}$; the morphisms are the $R$-module maps.

(iii) Let $\C$ be the category of chain complexes in $\M$.

(iv) A chain complex of left $R$-modules is called {\it finite} if it is 
finitely generated in each degree and zero in all but finitely many degrees, 
and {\it homotopy finite} if it is chain homotopy equivalent over $R$ to a 
finite 
chain complex.

(v) Let $\D$ be the full subcategory of $\C$ whose objects are the homotopy
finite chain complexes.
\end{definition}

Recall that, for a complex $C$ of left $R$-modules, $C^t$ is the complex of 
right $R$-modules obtained from $C$ by applying the involution of $R$.  Give
$C^t\otimes_R C$ the $\Z/2$ action that switches the factors.

Let $W$ be the standard resolution of $\Z$ by $\Z[\Z/2]$-modules.

\begin{definition}  A {\it quasi-symmetric complex of dimension $n$} is a 
pair $(C,\varphi)$, where $C$ is an object of $\D$ and $\varphi$ 
is a $\Z/2$-equivariant map 
\[
W\to C^t\otimes_R C
\]
of graded abelian groups which raises degrees by $n$.
\end{definition}

\begin{remark}
(i) An important example of a quasi-symmetric complex is the symmetric signature
of an object of $\A_{\pi,Z,w}$; see Section \ref{ss} for the definition.

(ii) A symmetric complex in the sense of Ranicki (\cite[Definition 
1.6(i)]{MR1211640}) is a 
quasi-symmetric complex for which $\varphi$ is a chain map.  

(iii) The concept of symmetric complex can be motivated as follows.  A symmetric
bilinear form on a vector space $V$ over a field $\mathbb F$ is a $\Z/2$
equivariant map $V\otimes V\to \mathbb F$.  This is the same thing as an
element of $\Hom_{\Z/2}(\Z, V^*\otimes_{\mathbb F} V^*)$.  In order to
generalize this concept to chain complexes we replace $V^*$ by $C$, $\mathbb F$
by $R$, and Hom by Ext; an element of the Ext group is represented by a 
symmetric complex.  Thus a symmetric complex is a homotopy version of a
symmetric bilinear form.
\end{remark}

\begin{definition}
\label{quasi}
We define a category 
$\A^R$ as follows.  The objects of $\A^R$ are the 
quasi-symmetric complexes.  Non-identity morphisms $(C,\varphi)\to
(C',\varphi')$ exist only when $\dim (C,\varphi)<\dim (C',\varphi')$, in which
case the morphisms are the
$R$-linear chain maps  $f:C\to C'$.  
\end{definition}

$\A^R$ is a balanced $\Z$-graded category, where $i$ takes $(C, \varphi)$ to
$(C, -\varphi)$ and $\emptyset_n$ is the $n$-dimensional object for which $C$ is zero in all degrees.

Next we must say what the $K$-ads with values in $\A^R$ are.  We will build up
to this gradually, culminating in Definition \ref{R-ad}.  

For a balanced pre $K$-ad $F$ we will use the notation
\[
F(\sigma,o)=(C_\sigma,\varphi_{\sigma,o}).
\]

\begin{definition}
(i) A map in $\M$ is a {\it strong monomorphism} if it is the inclusion of a
direct summand
\[
M\hookrightarrow M\oplus N
\]
with $M$ and $N$ in $\M$.

(ii)
A map of chain complexes over $R$ is a {\it cofibration} if it is a strong
monomorphism in each dimension.
\end{definition}

\begin{definition} 
\label{well2}
A functor $C$ from $\Cell^\flat(K)$ to chain complexes over $R$ is called {\it 
well-behaved} if the following conditions hold:

(a)
$C$ takes each morphism to a cofibration.

(b)
For each cell $\sigma$ of $K$, the map
\[\colim_{\tau\subsetneq\sigma}
C_\tau
\lra C_\sigma
\]
is a  cofibration.
\end{definition}

For a well-behaved functor $C$ we write
$C_{\partial\sigma}$  for $\colim_{\tau\subsetneq\sigma} C_\tau$.

For our next definition, recall Example \ref{gamma}.

\begin{definition} 
$F$ is {\it closed} if, for each $\sigma$, the map
\[
\cl(\sigma)\to \Hom(W,C_\sigma^t\otimes_R C_\sigma)
\]
which takes $\langle \tau,o\rangle$ to the composite
\[
W \labarrow{\varphi_{\tau,o}}
C_\tau^t\otimes C_\tau
\to
C_\sigma^t\otimes_R C_\sigma
\]
is a chain map.
\end{definition}

In particular, if $F$ is balanced and closed and $C$ is well-behaved then for 
each $\sigma$ the 
composite
\[
\bar{\varphi}:W\to C_\sigma^t\otimes_R C_\sigma \to 
(C_\sigma/C_{\partial\sigma})^t
\otimes_R 
C_\sigma
\]
is a chain map.

Let $\mathbf i$ be a generator of $H_0(W)$.  Recall Definition \ref{slant}.

\begin{definition}  
\label{R-ad}
$F$ is a {\it $K$-ad} if 

(a)
it is balanced and closed and $C$ is well-behaved, and

(b)
for each $\sigma$ the slant product with $\bar{\varphi}_*({\mathbf i})$ is an
isomorphism
\[
H^*(\Hom_R(C_\sigma,R))
\to
H_{\dim \sigma-\deg F-*}(C_\sigma/C_{\partial\sigma}).
\]
\end{definition}

We write $\ad^R(K)$ for the set of $K$ ads with values in $\A^R$.

\begin{remark}
When $K$ is a simplicial complex, a $K$-ad is almost the same thing as a
symmetric complex (\cite[Definition 3.4]{MR1211640}) 
in $\Lambda^*(K)$ (\cite[Definition 4.1 and Proposition 5.1]{MR1211640}).  The 
only difference is that in \cite{MR1211640} the splitting maps $C_\sigma\to
C_{\partial\sigma}$ of the underlying graded $R$-modules are part of the 
structure.
\end{remark}

\begin{thm}
\label{Oct16}
$\ad^R$ is an ad theory.
\end{thm}

For the proof we need a product operation on ads.  Recall the chain map
\[
\Delta:W\to W\otimes W
\]
from \cite[page 175]{MR560997}.

\begin{definition}
(i) For $i=1,2$, let $R_i$ be a ring with involution and let $(C^i,\varphi^i)$
be an object of 
$\A^{R_i}$.  Define 
\[
(C^1,\varphi^1)\otimes (C^2,\varphi^2)
\]
to be the following object of $\A^{R_1\otimes R_2}$:
\[
(C^1\otimes C^2,\psi),
\]
where $\psi$ is the composite 
\begin{multline*}
W\xrightarrow{\Delta} W\otimes W
\xrightarrow{\varphi^1\otimes\varphi^2}
((C^1)^t\otimes_{R_1} C^1) \otimes ((C^2)^t\otimes_{R_2} C^2) 
\\
\cong (C^1\otimes C^2)^t\otimes_{R_1\otimes R_2}\,(C^1\otimes C^2).
\end{multline*}

(ii) For $i=1,2$, suppose given 
a ball complex $K_i$ and a
pre $K_i$-ad $F_i$ of degree $k_i$ with values in $\A^{R_i}$. 
Define a pre $(K_1\times K_2)$-ad $F_1\otimes F_2$ with values in  
$\A^{R_1\otimes R_2}$
by 
\[
(F_1\otimes F_2)(\sigma\times\tau,o_1\times o_2)
=
i^{k_2\dim \sigma} F_1(\sigma,o_1)\otimes F_2(\tau,o_2).
\]
\end{definition}

\begin{lemma}
\label{prod2}
For $i=1,2$, suppose given 
a ball complex $K_i$ and a 
$K_i$-ad $F_i$ with values in $\A^{R_i}$. 
Then $F_1\otimes F_2$ is a $(K_1\times K_2)$-ad. \qed
\end{lemma}

\begin{proof}[Proof of \ref{Oct16}.]
We only need to verify parts (f) and (g) of Definition \ref{defad}.

The proof of part (f) is similar to the corresponding proof in Section
\ref{geom}.  
Let $F$ be a $K'$-ad with
\[
F(\sigma,o)=(C_\sigma,\varphi_{\sigma,o}).
\]
We need to define a $K$-ad $E$ which agrees with $F$ on each
residual subcomplex of $K$.
We may assume that $K$ is a ball complex structure for the $n$
disk with one $n$ cell $\tau$, and
that $K'$ is a subdivision of $K$ which agrees with $K$ on the boundary.
We only need to define $E$ on the top cell $\tau$ of $K$.  We define
$E(\tau,o)$ to be $( D_\tau, \kappa_{\tau,o})$, where
\begin{itemize}
\item
$D_\tau=\colim_{\sigma\in K'} C_\sigma$, and
\item
$\kappa_{\tau,o}$ is the sum of the composites
\[
W 
\labarrow{\varphi_{\sigma,o'}} 
C_\sigma^t\otimes_R C_\sigma
\to
D_\tau^t\otimes_R D_\tau,
\]
where $(\sigma,o')$ runs through the $n$-dimensional cells of $K'$ with 
orientations induced by $o$.
\end{itemize}
Then $E$ is closed by Proposition \ref{f7}(ii) and part (b) of Definition
\ref{geomad}
follows from Proposition \ref{superglue} below.

For part (g)
we need a preliminary definition.  
Define an $I$-ad $G$ with values in $\A^\Z$ as follows.
Let $0,1,I$ denote the three cells of $I$, with their standard
orientations. Define $G(0)$ to be $(\Z,\epsilon)$ where $\epsilon$ is
the augmentation, and similarly for 
$G(1)$.
Define $G(I)$ to be $(C_*(\Delta^1),\varphi)$, where $C_*$ denotes 
simplicial chains and $\varphi$ is the composite
\[
W\cong W\otimes{\mathbb Z}\stackrel{1\otimes\iota}{\to}
W\otimes C_1(\Delta^1)
\to
C_*(\Delta^1)\otimes C_*(\Delta^1);
\]
here $\iota$ is the element of $C_1(\Delta^1)$ represented by the identity map
and unlabeled arrow is the extended Alexander-Whitney map. 

Now let $F$ be a $K$-ad.  We define $J(F)$ on objects $(\sigma\times I,o\times
o')$ to be $F(\sigma,o)\otimes G(I,o')$.  The rest of the definition of $J(F)$
is analogous to the corresponding part of the proof of Theorem \ref{Sep30}.
$J(F)$ is an ad because it is isomorphic to $F\otimes G$.
\end{proof}

\begin{remark}
(i) The description of gluing in the proof just given, together with the 
proof of Lemma \ref{add}, shows that addition in the bordism groups of
$\ad^R$ is induced by direct sum.

(ii) A proof similar to that of Lemma \ref{j4} shows that a different choice 
of $\mathfrak S$ in Set Theoretic Prelude \ref{j8} gives a morphism of ad 
theories which induces an isomorphism on bordism groups.

(iii) The bordism groups of $\ad^R$ are the same (in fact they have the same
definition) as the groups $L^*({\mathbb A}^h(R))$ of \cite[Example 
1.11]{MR1211640}.
\end{remark}

\section{The symmetric signature}
\label{ss}

In this section we define a morphism of ad theories
\[
\Sig: \ad_{\pi,Z,w} \to \ad^{\Z[\pi]^w}
\]
called the {\it symmetric signature}
(here $\Z[\pi]^w$ denotes
$\Z[\pi]$ with the $w$-twisted involution \cite[page 196]{MR566491}).

As motivation, let us begin with the special case $(\pi,Z,w)=(e,*,1)$ (see 
Remark \ref{j2}).  An object of $\A_{e,*,1}$ is a pair $(X,\xi)$ and we define
\[
\Sig(X,\xi)=(S_*(X),\varphi_{X,\xi})
\]
where $\varphi_{X,\xi}$ is the composite 
\[
W\cong
W\otimes \Z
\xrightarrow{1\otimes \xi}
W\otimes S_*(X)
\to
S_*(X)\otimes S_*(X);
\]
the unlabeled arrow is the extended Alexander-Whitney map (see
\cite[Definition 2.10(a) and Remark 2.11(a)]{MR1969208}
for an explicit formula). Note that the third
condition in Set Theoretic Prelude \ref{j8} implies that 
$S_p(X)$ is in $\M$ for each $p$ (see Definition \ref{j7}) which is required for
Definition \ref{quasi}.

Next we consider a general triple $(\pi,Z,w)$.  Let $R$ denote $\Z[\pi]^w$. 
Let $(X,f,\xi)$ be an object of $\A_{\pi,Z,w}$ and recall that we write 
$\widetilde{X}$ for the pullback of $Z$ along $f$.  We define a map
\[
\varphi_{\widetilde{X},\xi}:W\to 
S_*(\widetilde{X})^t\otimes_R S_*(\widetilde{X})
\]
to be the composite
\begin{multline*}
W\cong
W\otimes \Z
\stackrel{1\otimes \xi}{\longrightarrow}
W\otimes
(\Z^w
\otimes_R
S_*(\widetilde{X}))
\cong
\Z^w
\otimes_R
(W \otimes S_*(\widetilde{X}))
\\
\to
\Z^w
\otimes_R
(S_*(\widetilde{X}) \otimes S_*(\widetilde{X}))
\cong
S_*(\widetilde{X})^t\otimes_R S_*(\widetilde{X}),
\end{multline*}
where the unlabeled arrow is induced by the extended Alexander-Whitney map.

Intuitively one would expect to define the symmetric signature by 
$\Sig(X,f,\xi)=(S_*(\widetilde{X}),\varphi_{\widetilde{X},\xi})$; the difficulty
with this is that $S_p(\widetilde{X})$ won't be an object of $\M$ in general.

To deal with this we redefine $\A_{\pi,Z,w}$.  By a {\it lifting function} for
$f:X\to Z/\pi$ we mean a function $\Phi$ that assigns to 
each map from a simplex to $X$ a lift to $\widetilde{X}$.

\begin{redefinition}
$\A_{\pi,Z,w}$
to be the category defined as follows.  The objects are quadruples
\[
(X,\,f:X\to Z/\pi,\,\xi\in S_*(X,{\mathbb Z}^f),\Phi),
\]
where
$X$ is a space in $\mathfrak X$ which has the homotopy type of a finite CW
complex and
$\Phi$ is a lifting function for $f$.
Non-identity morphisms $(X,f,\xi,\Phi)\to (X',f',\xi',\Phi')$ exist only when
$\dim \xi< \dim\xi'$, in which case the morphisms
are the maps $g:X\to X'$ such that
$f'\circ g=f$ and the diagram
\[
\xymatrix{
\Map(\Delta^p,X)
\ar[d]_\Phi
\ar[r]
&
\Map(\Delta^p,X')
\ar[d]^{\Phi'}
\\
\Map(\Delta^p, \widetilde{X})
\ar[r]
&
\Map(\Delta^p, \widetilde{X'})
}
\]
commutes for all $p$.
\end{redefinition}

It is easy to check that the forgetful functor from this version of 
$\A_{\pi,Z,w}$ to the previous version induces an isomorphism of bordism groups.

Now we define 
\[
\Sig: \A_{\pi,Z,w} \to \A^R.
\]
Let $(X,f,\xi,\Phi)$ be an object of $\A_{\pi,Z,w}$. 
For each $p\geq 0$, let $U_p$ be the set of maps $\Delta^p\to X$ and let $C_p$
be the free $R$-module $R\langle U_p\rangle$.
The lifting $\Phi$ gives an 
isomorphism of graded $R$-modules
\[
C_*\cong S_*(\widetilde{X}).
\]
We give $C_*$ the differential induced by this isomorphism and we define
\[
\psi:W\to C_*^t\otimes_R C_*
\]
to be the map determined by this isomorphism and 
the map $\varphi_{\widetilde{X},\xi}$ defined above.  Finally, we define
$\Sig(X,f,\xi,\Phi)$ to be $(C,\psi)$.

\begin{prop}
\label{j9}
$\Sig:\A_{\pi,Z,w}\to \A^R$ induces a morphism of ad theories 
\[
\Sig:\ad_{\pi,Z,w}\to \ad^R. 
\]
\qed
\end{prop}

\section{Example: Quadratic Poincar\'{e} ad theories}
\label{quad}

We use the notation of the previous section.

\begin{definition}
A {\it quasi-quadratic complex of dimension $n$} is a pair $(C,\psi)$ where 
$C$ is an object of $\D$ and $\psi$ is an element of
$( W\otimes_{\Z/2} (C^t\otimes_R C))_n$.
\end{definition}

\begin{definition}
\label{q}
We define a category 
$\A_R$ as follows.  The objects of $\A_R$ are the 
quasi-quadratic complexes.  
Non-identity morphisms $(C,\psi)\to (C',\psi')$ exist only when $\dim
(C,\psi)<\dim (C',\psi')$, in which case the morphisms are 
the $R$-linear chain maps  $f:C\to C'$.  
\end{definition}

$\A_R$ is a balanced $\Z$-graded category, where $i$ takes $(C, \psi)$ to
$(C, -\psi)$ and $\emptyset_n$ is the $n$-dimensional object for which $C$
is zero in all degrees.

A balanced pre $K$-ad $F$ has the form
\[
F(\sigma,o)=(C_\sigma,\psi_{\sigma,o}).
\]

\begin{definition}
$F$ is {\it closed} if, for each $\sigma$, the map
\[
\cl(\sigma)
\to
W\otimes_{\Z/2} (C_\sigma^t\otimes_R C_\sigma)
\]
which takes $\langle \tau,o\rangle$ to the image of
$\psi_{\tau,o}$ is a chain map. 
\end{definition}

Next we define a nonpositively graded complex of $\Z/2$-modules
\[
V_0\to V_{-1} \to \cdots
\]
by letting
\[
V_{-n}=\Hom_{\Z/2}(W_n,\Z[\Z/2]).
\]
There is an isomorphism
\[
W\otimes_{\Z/2} (C^t\otimes_R C)
\cong
\Hom_{\Z/2}(V,C^t\otimes_R C).
\]
The composite
\[
N:W\to \Z\to V
\]
induces a homomorphism
\[
N^*:W\otimes_{\Z/2} (C^t\otimes_R C)
\to
\Hom_{\Z/2}(W,C^t\otimes_R C)
\]
called the {\it norm map}.  We write $\N$ for the functor
\[
\A_R\to \A^R
\]
which takes $(C,\psi)$ to $(C,N^*(\psi))$.

\begin{definition}
$F\in\pre_R(K)$ is a {\it $K$-ad} if

(a)
it is balanced and closed and $C$ is well-behaved, and

(b)
$\N\circ F$ is a $K$-ad.
\end{definition}

\begin{thm}
\label{n13}
$\ad_R$ is an ad theory.
\end{thm}

For the proof we need a product operation.  Ranicki (\cite[pages 
174--175]{MR560997}) defines a chain map 
\[
\Delta:V\to W\otimes V.
\]

\begin{definition}
(i) Let $R_1$ and $R_2$ be rings with involution.  
Let $(C,\varphi)$ be an object of $\A^{R_1}$ and let
$(D,\psi)$ be an object of $\A_{R_2}$.  Define
\[
(C,\varphi)\otimes (D,\psi)
\]
to be the following object of $\A_{R_1\otimes R_2}$:
\[
(C\otimes D,\omega),
\]
where $\omega$ is the element of 
\[
W\otimes_{\Z/2}((C\otimes D)^t
\otimes_{R_1\otimes R_2}
(C\otimes D))
\]
corresponding to the composite
\begin{multline*}
V\labarrow{\Delta}W\otimes V
\labarrow{\varphi\otimes \psi}
(C^t\otimes_{R_1} C)
\otimes
(D^t\otimes_{R_2} D)
\cong
(C\otimes D)^t
\otimes_{R_1\otimes R_2}
(C\otimes D).
\end{multline*}

(ii) Suppose given ball complexes $K_1$ and $K_2$, a pre 
$K_1$-ad $F_1$ of degree $k_1$ with values in $\A^{R_1}$, and a pre $K_2$-ad 
$F_2$ of degree $k_2$ with values
in $\A_{R_2}$.
Define a pre $(K_1\times K_2)$-ad $F_1\otimes F_2$ with values in
$\A_{R_1\otimes R_2}$
by
\[
(F_1\otimes F_2)(\sigma\times\tau,o_1\times o_2)
=
i^{k_2\dim\sigma}F_1(\sigma,o_1)\otimes F_2(\tau,o_2).
\]
\end{definition}

\begin{lemma}
\label{n14}
Suppose given ball complexes $K_1$ and $K_2$, a $K_1$-ad $F_1$ with values in 
$\A^{R_1}$, and a $K_2$-ad $F_2$ with values in $\A_{R_2}$.  Then $F_1\otimes 
F_2$ is a $(K_1\times K_2)$-ad. \qed
\end{lemma}

\begin{proof}[Proof of Lemma \ref{n14}]
First observe that the set of homotopy classes
of chain maps from $W$ to a chain complex $A$ is the same as $H_0(A)$.
It follows that the diagram
\[
\xymatrix{
W
\ar[r]^N
\ar[d]_\Delta
&
V
\ar[d]^\Delta
\\
W\otimes W
\ar[r]^-{1\otimes N}
&
W\otimes V
}
\]
homotopy commutes.
The result follows from this and Lemma \ref{prod2}.
\end{proof}

The proof of Theorem \ref{n13} is now completely analogous to that of 
Theorem \ref{Oct16}.

\section{Gluing}
\label{glue}

Our goal in this section is to prove a result (Proposition \ref{superglue})
which completes the proofs of Theorems \ref{Oct6}, \ref{Oct16}, and \ref{n13}.
First we need some terminology.

Let $R$ be a ring with involution.  

Recall Definition \ref{j7}(v).
Let $\A$ be the $\Z$-graded category defined as follows.
The objects of dimension $n$ are pairs 
$(C,\zeta)$, where $C$ is an object of $\D$ and $\zeta$ is an 
$n$-dimensional element of $C^t\otimes_R C$. 
Non-identity morphisms $(C,\zeta)\to (C',\zeta')$ exist only when $\dim
(C,\zeta)< \dim (C',\zeta')$, in which case the
morphisms 
are the $R$-linear chain maps  $f:C\to C'$.
$i$ takes $(C,\zeta)$ to $(C,-\zeta)$ and $\emptyset_n$ is the $n$-dimensional
object for which $C$ is 0 in all degrees.

A balanced pre $K$-ad $F$ with values in $\A$ is {\it closed} if for each
$\sigma$ the elements $\zeta_{\tau,o}$ determine a chain map $\cl(\sigma)\to 
C_\sigma^t\otimes_R C_\sigma$.  $F$ is a {\it $K$-ad} if it is balanced and 
closed, $C$ is well-behaved, and the slant product with $\zeta_{\sigma,o}$ is 
an isomorphism
\[
H^*(\Hom_R(C_\sigma,R))\to H_{\dim \sigma-\deg F-*} 
(C_\sigma/C_{\partial\sigma})
\]
for each $\sigma$.

\begin{definition}
A {\it Poincar\'e pair} is a morphism $(C,\zeta)\to (D,\omega)$ in $\A$ with 
the property that the pre $I$-ad $G$ defined by $G(1)=(C,\zeta)$, 
$G(I)=(D,\omega)$ and $G(0)=\emptyset$ is an ad.
\end{definition}

\begin{definition} Let $K$ be a ball complex and 
$C:\Cell^\flat(K)\to\D$ a well-behaved functor.  Define $C_K\in \D$ 
to be $ \colim_{\sigma\in K}\, C_\sigma $.
\end{definition}

Now let $(L,L_0)$ be a ball complex pair such that
$|L|$ is an orientable homology manifold with boundary $|L_0|$, and fix an 
orientation for $|L|$. (For the proofs of Theorems \ref{Oct6}, \ref{Oct16}, and
\ref{n13} we only need the special case where $|L|$ is a PL ball).

\begin{definition}
Let $C:\Cell^\flat(L)\to\D$ be a well behaved functor and let 
\[
\nu:\cl\to C
\]
be a natural transformation.  Denote the value of $\nu$ on $\langle
\sigma,o\rangle$ by $\nu_{\sigma,o}$. Define $\nu_L\in C_L$ (resp., 
$\nu_{L_0}\in
C_{L_0}$) to be
\[
\sum_{(\sigma,o)}\, \nu_{\sigma,o},
\]
where $(\sigma,o)$ runs through the top-dimensional cells of $L$ (resp.,
$L_0$) oriented compatibly with $|L|$.
\end{definition}

\begin{prop}
\label{superglue}
Let $F$ be an $L$-ad and write $F(\sigma,o)=(C_\sigma, \zeta_{\sigma,o})$.
Then
\[
(C_{L_0},\zeta_{L_0})
\to
(C_L,\zeta_L)
\]
is a Poincar\'e pair.
\end{prop}

\begin{remark}
\label{f4}
The corresponding statements for the ad theories $\ad_{\pi,Z,w}$, $\ad^R$ and
$\ad_R$ are consequences of this.
\end{remark}

The rest of this section is devoted to the proof of Proposition
\ref{superglue}.  What we need to show is that the slant product
\[
H^*(\Hom_R(C_L,R))\to H_{\dim |L|-\deg F-*}(C_L/C_{L_0})
\]
is an isomorphism.

The first step in the proof is to give an alternate description of
$H_{\dim |L|-*}(C_L/C_{L_0})$.

Let 
\[
B:\Cell^\flat(L)\to \D
\]
be a well-behaved functor and consider the chain complex 
\[
\Nat(\cl,B)
\]
of natural transformations of graded abelian groups; the differential is 
given by
\[
\partial(\nu)=\partial\circ \nu-(-1)^{|\nu|}\nu\circ \partial.
\]
Define 
\[
\Phi:\Nat(\cl,B)\to 
B_L/B_{L_0}
\]
by 
\[
\Phi(\nu)=\nu_L.
\]
Then $\Phi$ is a chain map by Proposition \ref{f7}(ii); note that $\Phi$ 
increases degrees by $\dim |L|$.

\begin{lemma} {\rm (cf.\ \cite[Digression 3.11]{MR1026874})}
\label{f3}
$\Phi$ induces an isomorphism 
\[
H_*(\Nat(\cl,B))
\to
H_{*+\dim |L|}(B_L/B_{L_0})
\]
for every well-behaved $B:\Cell^\flat(L)\to \D$.
\end{lemma}

The proof is deferred to the end of this section.

Continuing with the proof of Proposition \ref{superglue}, we observe that
\begin{equation}
\label{f10}
\Hom_R(C_L,R)=\Nat_R(C,\underline{R}),
\end{equation}
where $\Nat_R$ denotes the chain complex of natural transformations of graded
$R$-modules and $\underline{R}$ denotes the constant functor with value $R$.
There is a slant product
\[
\Upsilon:\Nat_R(C,\underline{R})
\to
\Nat(\cl,C)
\]
which takes $\nu$ to the composite
\[
\cl\labarrow{\zeta}
C^t\otimes_R C
\labarrow{1\otimes \nu}
C^t\otimes_R \underline{R}
=
C^t
\]
(note that $C^t$ and $C$ are the same as functors to graded abelian groups).
The diagram 
\[
\xymatrix{
H_{-*}(\Nat_R(C,\underline{R}))
\ar[r]^-{H_*\Upsilon}
\ar[d]_=
&
H_{-*-\deg F}(\Nat(\cl,C))
\ar[d]_\cong^{\text{Lemma \ref{f3}}}
\\
H^*(\Hom_R(C_L,R))
\ar[r]
&
H_{\dim |L|-\deg F-*}(C_L/C_{L_0})
}
\]
commutes, so to prove Proposition \ref{superglue} it suffices to show that 
$\Upsilon$ is a homology isomorphism.

Next observe that 
\[
\Nat_R(C,\underline{R})
=
\lim_{\sigma\in L}\, \Nat_R(C|_{\Cell^\flat(\sigma)},\underline{R})
\]
and that
\[
\Nat(\cl,C)=\lim_{\sigma\in L}\, \Nat(\cl,C|_{\Cell^\flat(\sigma)}).
\]
Moreover, the natural maps
\[
\lim_{\sigma\in L}\, \Nat_R(C|_{\Cell^\flat(\sigma)},\underline{R}) 
\to 
\holim_{\sigma\in L}\, \Nat_R(C|_{\Cell^\flat(\sigma)},\underline{R})
\]
and
\[
\lim_{\sigma\in L}\, \Nat(\cl,C|_{\Cell^\flat(\sigma)})
\to
\holim_{\sigma\in L}\, \Nat(\cl,C|_{\Cell^\flat(\sigma)})
\]
are homology isomorphisms by 
\cite[Theorem 19.9.1(2)]{MR1944041} (using the fact that $C$ and $\cl$ are
well-behaved).
Thus there are spectral sequences 
\[
\text{lim}^p H_q(\Nat_R(C|_{\Cell^\flat(\sigma)},\underline{R}))
\Rightarrow
H_{q-p}(\Nat_R(C,\underline{R}))
\]
and 
\[
\text{lim}^p H_q(\Nat(\cl,C|_{\Cell^\flat(\sigma)}))
\Rightarrow
H_{q-p}(\Nat(\cl,C))
\]
(see 
\cite[Section XI.7]{MR0365573} for the construction; note that in the category
of chain complexes $H_*$ plays the role of $\pi_*$).
By \cite[Proposition XI.6.2]{MR0365573} we have
$\text{lim}^p=0$ for $p>\dim|L|$, so 
these spectral sequences converge strongly. 

The slant products
\[
\Upsilon|_\sigma: \Nat_R(C|_{\Cell^\flat(\sigma)},\underline{R})
\to
\Nat(\cl,C|_{\Cell^\flat(\sigma)})
\]
give a map of inverse systems and hence a map of spectral sequences. 
By equation \eqref{f10} and Lemma \ref{f3} the maps
\[
H_*(\Upsilon|_\sigma): 
H_*(\Nat_R(C|_{\Cell^\flat(\sigma)},\underline{R}))
\to
H_{*-\deg F} (\Nat(\cl,C|_{\Cell^\flat(\sigma)}))
\]
agree up to isomorphism with the slant products
\[
H^*(\Hom_R(C_\sigma,R))
\to
H_{\dim\sigma-*}(C_\sigma/C_{\partial\sigma})
\]
which are isomorphisms because $F$ is an ad.  Thus the $\Upsilon|_\sigma$ 
give an isomorphism at $E^2$. Since $\lim \Upsilon|_\sigma$ is $\Upsilon$ we 
see that $H_*(\Upsilon)$ is an isomorphism, which completes the proof of 
Proposition \ref{superglue}.

\begin{proof}[Proof of Lemma \ref{f3}]
The proof is similar to the proof of \cite[Digression 3.11]{MR1026874}.

First of all, the proof of \cite[Lemma 3.4]{MR1026874} adapts to our situation
to show that $B$ is weakly equivalent to a well-behaved functor which is 
finitely generated (that is, one which takes each $\sigma$ to a finitely 
generated complex).  The source and target of $\Phi$ both preserve weak
equivalences (see the argument at the top of page 71 in \cite{MR1026874}) and
so we may assume that $B$ is finitely generated.

Because $B$ is well-behaved, the source and target of $\Phi$ both have the 
property that they take short exact sequences of well-behaved functors to 
short exact sequences.  We give $B$ a decreasing filtration by letting the 
$i$-th filtration $B[i]$ take $\sigma$ to the sum of the images of 
$B_{\sigma'}\to B_{\sigma}$ with $\sigma'\subset\sigma$ and
$\dim\sigma'\geq i$.  Then the sequence
\[
0\to B[i+1]\to B[i]\to B[i]/B[i+1]\to 0 
\]
is a short exact sequence of well-behaved functors, so it suffices (by 
induction on $i$) to show that the lemma is true for the quotients 
$B[i]/B[i+1]$ when $0\leq i\leq n$.  Now each quotient $B[i]/B[i+1]$ is a 
direct sum
\[
\oplus_{\dim\rho=i}\, B[\rho]
\]
where $B[\rho]_\sigma$ is the image of $B_\rho\to B_\sigma$ if
$\rho\subset\sigma$ and 0 otherwise.  Next we give  $B[\rho]$ an
increasing filtration by letting the $j$-th filtration $B[\rho,j]$ be the 
functor which takes $\sigma$ to the part of $B[\rho]_\sigma$ in dimensions 
$\leq j$.
The sequence
\[
0\to B[\rho,j]\to B[\rho,j+1]\to B[\rho,j+1]/B[\rho,j]\to 0
\]
is a short exact sequence of well-behaved functors for each $j$.  Since $B$ is
finitely generated, it suffices to show that the lemma is true for the 
quotients $B[\rho,j+1]/B[\rho,j]$.  

Fix $\rho$ and $j$.  To lighten the notation let us denote
$B[\rho,j+1]/B[\rho,j]$ by $A$.

The functor $A$ takes $\rho$ to a chain complex which
consists of an abelian group (call it $\fA$) in dimension $j$ and 0 in all
other dimensions.   It takes every cell containing $\rho$ to this same chain
complex.  Let $M$ be the subcomplex of $L$ consisting of all cells which 
contain $\rho$ and their faces.
Let $N$ be the subcomplex of $M$ consisting of all cells which do not contain
$\rho$.  Then the chain complex 
$\Nat(\cl,A)$ is isomorphic to the 
cellular cochain complex 
$C^{j-*}(M,N;\fA)$.  Next we use results from the Appendix.  
Proposition \ref{f8} gives a ball complex structure on 
the pair $(|\st(\hat\rho)|,|\lk(\hat\rho)|)$.  
There is a bijection between the cells of the cells of  $M$ which are not in
$N$ and the cells of 
$|\st(\hat\rho)|$ which are not in $|\lk(\hat\rho)|$; this bijection preserves 
incidence numbers and therefore induces an isomorphism
\[
C^*(M,N;\fA)\cong C^*(|\st(\hat\rho)|,|\lk(\hat\rho)|;\fA).
\]

Now 
\[
H^*(|\st(\hat\rho)|,|\lk(\hat\rho)|;\fA)=H^*(|L|,|L|-\hat\rho;\fA).
\]
Thus $H^*(|\st(\hat\rho)|,|\lk(\hat\rho)|;\fA)$ is 0 if 
$*\neq \dim |L|$ or if $\rho$ is in $L_0$.  In the remaining case, we note that
$|\st(\hat\rho)|$ is a homology manifold with boundary $|\lk(\hat\rho)|$, and
thus (by Proposition \ref{f7}(i)) the map 
\[
H^*(|\st(\hat\rho)|,|\lk(\hat\rho)|;\fA)
\to
\fA
\]
which takes a cocycle to its value on the sum of the top-dimensional cells of
$|\st(\hat\rho)|$ is an isomorphism.

To sum up, we have shown that if $*\neq j-\dim |L|$ or if $\rho$ is in $L_0$
then
$H_*(\Nat(\cl,A))$ is 0, and that the map
\[
H_{j-\dim |L|}(\Nat(\cl,A))
\to
\fA
\]
which takes $\nu$ to $\nu_L$ is an isomorphism.  

Now if $\rho$ is not in $L_0$ then $A_{L_0}$ is 0, so 
$H_*(A_L/A_{L_0})$ is 
$\fA$ when $*=j$ and 0 otherwise.  This proves the lemma in this case.

If $\rho$ is in $L_0$ then 
$A_L=A_{L_0}$, so the domain and 
target of the map in the lemma are both 0.
\end{proof}

\section{Functoriality}
\label{fun}

We begin by considering symmetric Poincar\'e ad theories.  

Recall Definition \ref{j7}.  We will let $R$ vary in this section so we write
$\M_R$ instead of $\M$.  We want to make $\M_R$ a functor of $R$.  It would be
natural to attempt to do this as follows: if $p:R\to S$ is a homomorphism of 
rings with involution and $M$ is an object of $\M_R$ define $p_*M= S\otimes_R 
M$. Unfortunately this cannot be correct for two reasons.  First, $S\otimes_R 
M$ is not of the form $R\langle A\rangle$ and hence is not in $\M_S$.  A more
serious difficulty is that if $q:S\to T$ is another homomorphism of 
rings with involution then $(qp)_*M$ is isomorphic to, but not equal to, 
$q_*p_*M$.

A similar problem arises in algebraic $K$-theory, and Blumberg and Mandell have
given a solution (see the proof of Theorem 8.1 in \cite{BM})
which also works for our situation.  We redefine $\M_R$ 
by letting its objects be the sets in $\mathfrak S$; the morphisms are still
the $R$-module maps 
$R\langle A\rangle$ to $R\langle B \rangle$, which we think of as (possibly
infinite) matrices with values in $R$.
Given $p:R\to S$ and an object $A$ of $\M_R$ we define $p_*A=A$; for a morphism
$\alpha:A\to B$ in $\M_R$ we let $p_*\alpha$ be the matrix obtained by applying
$p$ to the entries of the matrix $\alpha$.

It now follows that $\A_R$ and $\ad^R$ are functors of $R$.  Quadratic
Poincar\'e ad theories can be dealt with in a similar way.

\begin{notation}
(i) Let $\R$ be the category of rings with involution.

(ii) Let $\ad_{\mathrm{sym}}$ be the 
functor from $\R$ to the category of ad-theories that takes $R$ to $\ad^R$ .

(iii) Let $\ad_{\mathrm{quad}}$ be the
functor from $\R$ to the category of ad-theories that takes $R$ to $\ad_R$.
\end{notation}

Next we consider functoriality of $\ad_{\pi,Z,w}$ (as redefined in Section
\ref{ss}). 

\begin{definition}
\label{Nov7.2}
(i) Let $\T$ be the category whose objects are the triples $(\pi,Z,w)$;  
the morphisms from 
$(\pi,Z,w)$ to $(\pi',Z',w')$ are pairs $(h,g)$, where $h:\pi\to \pi'$ is a 
homomorphism with $w=w'\circ h$ and $g$ is a $\pi$-equivariant map $Z\to Z'$.

(ii) Let $\rho:\T\to \R$ be the functor which takes $(\pi,Z,w)$ to
$\Z[\pi]$ with the $w$-twisted involution.
\end{definition}

For an object $(X,f,\xi,\Phi)$ of $\ad_{\pi,Z,w}$ we write $f^*Z$
for the pullback of $Z$ along $f$ (this was denoted $\widetilde{X}$ in earlier
sections).  A morphism in $\T$ induces a functor $\A_{\pi,Z,w}\to 
\A_{\pi',Z',w'}$ by taking $(X,f,\xi,\Phi)$ to $(X, \bar{g}f, \eta, \Psi)$, 
where $\bar{g}$ is the map $Z/\pi\to Z'/\pi'$ induced by $g$, $\eta$ 
corresponds to $\xi$ under the isomorphism 
\[
S_*(X,\Z^f)\cong S_*(X,\Z^{\bar{g}f})
\]
(see Definition \ref{Nov7.1}),
and $\Psi$ is determined by $\Phi$ together with the canonical map 
$f^*Z\to f^*\bar{g}^*Z'$.

With these definitions $\ad_{\pi,Z,w}$ is a functor of $(\pi,Z,w)$.  

\begin{notation}
Let
$\ad_{\mathrm{geom}}$ be the functor from $\T$ to the category of
ad-theories that takes $(\pi,Z,w)$ to $\ad_{\pi,Z,w}$.
\end{notation}

Finally, we note that 
$\Sig$ (defined in Section \ref{ss}) is a natural
transformation from $\ad_{\mathrm{geom}}$ to 
$\ad_{\mathrm{sym}}\circ \rho$.

\section{The cohomology theory associated to an ad theory}
\label{cohomology section}

Fix an ad theory.

For a ball complex $K$ with a subcomplex $L$, we will say that two 
elements $F,G$ of $\ad^k(K,L)$ are {\it bordant} if there is a 
$(K\times I,L\times I)$-ad which restricts to $F$ on $K\times 0$ and $G$ on 
$K\times 1$.

\begin{definition}
Let $T^k(K,L)$ be the set of bordism classes in $\ad^k(K,L)$.
\end{definition}

\begin{remark}
(i) $T^{k}(*)$ is the same as $\Omega_{-k}$.

(ii) For the ad theory in Example \ref{gamma}, $T^k(K,L)$ is $H^k(K,L;C)$.
\end{remark}

Our goal in this section is to show that $T^*$ is a cohomology theory.

We will define addition in $T^k(K,L)$ using the method of Section 
\ref{groups}.  First we need a generalization of the functor $\kappa$.

\begin{definition}
\label{n7}
Let 
\[
\kappa:\Cell(I\times K,(\{0,1\}\times K)\cup (I\times L))
\to
\Cell(K,L)
\]
be the isomorphism of categories which takes 
$I\times(\sigma,o)$ (where $I$
is given its standard orientation) to $(\sigma,o)$.  
\end{definition}

\begin{remark}
\label{n26}
$\kappa$ is incidence-compatible (Definition \ref{n25}(i)) so it induces a
bijection
\[
\kappa^*:\ad^k(K,L)\to \ad^{k+1}(I\times K,(\{0,1\}\times K)\cup (I\times L))
\]
by part (e) of Definition \ref{defad}.
\end{remark}

Now let $M$ and $M'$ be the ball complexes defined in Section 
\ref{groups}.  Lemma \ref{add}
generalizes to show that, given $F,G\in \ad^k(K,L)$, there is an
$H\in\ad^{k+1}(M\times K,M\times L)$ such that 
$(\lambda_1 \times \id)^*H=\kappa^*F$,
$(\lambda_2 \times id)^*H=\kappa^*G$, and $(\lambda_3\times\id)^*H$ and 
$(\lambda_4\times \id)^*H$ 
are trivial.  Then we define $[F]+[G]$ to be
\[
[(\kappa^{-1})^*(\lambda_5\times \id)^*H].
\]
The proof that this is well-defined and that 
$T^k(K,L)$ is an abelian group is the same as in Section \ref{groups}.

Next we show that $T^k$ is a homotopy functor.

Using the notation of \cite[page 5]{MR0413113}, let $Bi$ be the category whose 
objects are pairs of ball complexes and whose morphisms are composites of 
inclusions of subcomplexes and isomorphisms.  Let $Bh$ be the category with 
the same objects whose morphisms are homotopy classes of continuous maps of 
pairs.  

\begin{prop}
\label{n11}
For each $k$, the functor $T^k:Bi\to \Ab$ factors uniquely through $Bh$.
\end{prop}

The functor $Bh\to \Ab$ given by the lemma will also be denoted by $T^k$.

For the proof of Proposition \ref{n11} we need a preliminary fact.

\begin{definition}(cf.\ \cite[page 5]{MR0413113})\
An inclusion of pairs $(K_1,L_1)\to (K,L)$ is an {\it elementary expansion} if

(a)
$L_1=L\cap K_1$, 

(b)
$K$ has exactly two cells (say $\sigma$ and $\sigma'$) that 
are not in $K_1$, with $\dim \sigma'=\dim \sigma-1$ and 
$\sigma'\subset\partial \sigma$, and 

(c)
$\sigma$ and $\sigma'$ are either both in $L$ or both not in $L$.
\end{definition}

\begin{lemma}
\label{Dec8}
If $(K_1,L_1)\to (K,L)$ is an elementary expansion then the restriction
\[
\ad^k(K,L)\to \ad^k(K_1,L_1)
\]
is onto.
\end{lemma}

The proof is deferred to the end of this section.

\begin{proof}[Proof of Proposition \ref{n11}]
The functor $\ad^k$ satisfies axioms E and G on page 15 of 
\cite{MR0413113}; axiom E is Lemma \ref{Dec8} and axiom G is part (d) of
Definition \ref{defad}.  Now Proposition I.6.1 and Theorem I.5.1 of
\cite{MR0413113} show that 
\[
T^k:Bi\to \Set
\]
factors uniquely to give a functor $T^k:Bh\to \Set$.
Specifically (with the notation of Definition \ref{ball}) if $f:(|K'|,|L'|)\to 
(|K|,|L|)$ is a map of pairs then $T^k(f)$ is
defined to be $T^k(g)^{-1}T^k(h)$, where $g$ and $h$ are certain morphisms in
$Bi$.  But then $T^k(f)$ is a homomorphism, so we obtain a functor 
$T^k:Bh\to \Ab$.
\end{proof}

Next we observe that excision is an immediate consequence of part (e) of
Definition \ref{defad}.

The first step in constructing the connecting homomorphism is to construct a
suitable suspension isomorphism.

\begin{lemma}
$\kappa^*$ induces an isomorphism
\[
T^k L \to T^{k+1}(I\times L,\{0,1\}\times L).
\] 
\end{lemma}

\begin{proof}
$\kappa^*$ is a bijection by Remark \ref{n26}.
To see that it is a homomorphism, let $F,G\in \ad^k(L)$ and let $H\in
\ad^{k+1}(M\times K,M\times L)$ be as in the definition of addition.
Let 
\[
\theta:
\Cell(M\times I\times K,(M\times\{0,1\}\times K)\cup (M\times I\times L))
\to
\Cell(M\times K,M\times L)
\]
be the evident isomorphism. Then $\theta^*(H)$ is an $(M\times I \times K)$-ad
with the property that $(\lambda_1 \times \id)^*\theta^*(H)=\kappa^*\kappa^*F$,
$(\lambda_2 \times \id)^*\theta^*(H)=\kappa^*\kappa^*G$, and
$(\lambda_3\times\id)^*\theta^*(H)$ and 
$(\lambda_4\times\id)^*\theta^*(H)$ are trivial.  Thus $[\kappa^*F]+[\kappa^*G]$
is $[(\kappa^{-1})^*(\lambda_5\times \id)^*\theta^*(H)]$, which simplifies to
$[(\lambda_5\times \id)^*H]$, and this is $\kappa^*[F+G]$.
\end{proof}

\begin{remark}
The statement of the lemma might look strange in view of the fact that, for a
space $X$, $(I\times X)/(\{0,1\}\times X)$ is homotopic to $\Sigma X\vee S^1$ 
rather than $\Sigma X$.  But if $E$ is a cohomology theory 
then $E^k(X)\cong \tilde{E}^k(X\vee S^0)$, so the lemma agrees with the
expected behavior of the suspension map for unreduced cohomology theories.
\end{remark}

Now observe that excision gives an isomorphism
\[
T^k(I\times L, \{0,1\}\times L)
\labarrow{\cong}
T^k((1\times K)\cup (I\times L), (1\times K) \cup (0 \times L))
\]
(where $(1\times K) \cup (I\times L)$ is thought of as a subcomplex of $I\times 
K$) and that the map 
\[
(|(1\times K)\cup (I\times L)|, |(1\times K) \cup (0 \times L)|)
\to
(|I\times K|,|(1\times K) \cup (0 \times L)|)
\]
is a homotopy equivalence of pairs.  It follows that the restriction map 
\[
T^k(I\times K,(1\times K) \cup (0 \times L))
\to
T^k(I\times L, \{0,1\}\times L)
\]
is an isomorphism.

\begin{definition}
The connecting homomorphism
\[
T^k(L)\to T^{k+1}(K,L)
\]
is the negative of the composite
\[
T^k(L)\labarrow{\kappa^*}
T^{k+1}(I\times L, \{0,1\}\times L)
\stackrel{\cong}{\leftarrow}
T^{k+1}(I\times K,(1\times K) \cup (0 \times L))
\to
T^{k+1}(K,L)
\]
where the last map is induced by the inclusion 
\[
(0\times K,0\times L)\to
(I\times K,(1\times K) \cup (0 \times L)).
\]
\end{definition}

For an explanation of the sign see the proof of Proposition \ref{n12}(ii).

\begin{thm} 
\label{theorem 4.2} 
$T^*$ is a cohomology theory.
\end{thm}

\begin{proof}
It only remains to verify that the sequence
\[
T^{k-1}K\to T^{k-1}L\to T^k(K,L)\to T^k K \to T^k L
\]
is exact for every pair $(K,L)$.

{\it Exactness at $T^k(K)$.} We prove 
more generally that the sequence
\[
T^k(K,L)\to T^k(K,M) \to T^k(L,M)
\]
is exact for every triple $M\subset L\subset K$.

Clearly the composite 
$T^k(K,L)\to T^k(K,M) \to T^k(L,M)$ is trivial.  On the other hand, if 
$[F]\in T^k(K,M)$ maps to 0 in $T^k(L,M)$, then there is a bordism $H\in
\ad^k(L\times I,M\times I)$ from $F|_L$ to $\emptyset$.  We obtain an ad 
\[
H'\in
\ad^k((K\times 1)\cup (L\times I), M\times I)
\]
by letting $H'$ be $F$ on $K\times 
1$ and $H$ on $L\times I$.  The inclusion 
\[
(K\times 1)\cup (L\times I)\to K\times I
\]
is a composite of elementary expansions, so by Lemma \ref{Dec8} there 
is an $H''\in\ad^k(K\times I,M\times I)$ which restricts to $H'$.  But now 
$H''|_{K\times 0}$ is in $\ad^k(K,L)$ and is bordant to $F$, so $[F]$ is in the
image of $T^k(K,L)$.

{\it Exactness at $T^k(K,L)$.} The composite
\[
T^k(I\times K,(1\times K) \cup (0 \times L))
\to
T^k(K,L)
\to
T^k(K)
\]
takes $F$ to $F|_{0\times K}$, but this is bordant to $F|_{1\times K}$ which is
0.  It follows that the composite
\[
T^{k-1}L\to T^k(K,L)\to T^k(K)
\]
is trivial.  On the other hand, if $F\in \ad^k(K,L)$ becomes 0 in $T^k(K)$
then there is an $H\in\ad^k(I\times K,(1\times K) \cup (0 \times L))$ 
with $H|_{0\times K}=F$.  Thus $F$ is in the image of 
\[
T^k(I\times K,(1\times K) \cup (0 \times L))
\to
T^k(K,L)
\]
and hence in the image of the connecting homomorphism.

{\it Exactness at $T^{k-1}L$.}  
The composite
\[
T^{k-1}K\to T^{k-1}L\to T^k(K,L)
\]
is equal to the composite
\[
T^{k-1}K
\labarrow{\kappa^*}
T^k(I\times K,\{0,1\}\times K)
\to
T^k(I\times K,(1\times K)\cup (0\times L))
\to
T^k(K,L)
\]
and the composite of the last two maps is clearly trivial.
On the other hand, suppose that $x\in T^{k-1}(L)$ maps trivially to
$T^k(K,L)$.  By definition of the connecting homomorphism, there is a $y\in
T^k(I\times K,(1\times K)\cup (0\times L))$ such that $y$ restricts to
$\kappa^*x$ in $T^k(I\times L,\{0,1\}\times L)$ and to 0 in $T^k(0\times
K,0\times L)$.  Since the restriction map 
\[
T^k(\{0,1\}\times K,(1\times K)\cup (0\times L))
\to
T^k(0\times K,0\times L)
\]
is an isomorphism by excision, we see that $y$ restricts trivially to
$T^k(\{0,1\}\times K,(1\times K)\cup (0\times L))$.  Now the exact sequence of
the triple 
\[
(1\times K)\cup (0\times L)\subset \{0,1\}\times K\subset I\times K
\]
implies that there is a $z\in T^k(I\times K,\{0,1\}\times K)$ that restricts to
$y$.  Then $z$ restricts to $\kappa^*x$ in $T^k(I\times L,\{0,1\}\times L)$
and therefore $(\kappa^*)^{-1}z\in T^k(K)$ restricts to $x$.
\end{proof}

\begin{proof}[Proof of \ref{Dec8}.]
Let $F\in \ad^k(K_1,L_1)$.
Let $\sigma'$ and $\sigma$ be as in the definition of elementary expansion.
If $\sigma'$ and $\sigma$ are in $L$ then we can extend
$F$ to $\Cell(K,L)$ by letting it take $\sigma'$ and $\sigma$ to 
$\emptyset$.  So assume that $\sigma'$ and $\sigma$ are not in $L$.  
Let $A$ be the sub-ball-complex of $K$ which is the union of the cells of 
$\partial\sigma$ other than $\sigma'$.  It suffices to show that the 
restriction of $F$ to $\Cell(A)$ extends to $\Cell(\sigma)$. By Theorem 3.34 
of \cite{MR665919}, the pair $(\sigma,\sigma')$ is PL isomorphic to the pair 
$(D^n,S^{n-1}_-)$ (where $n$ is the dimension of $\sigma$, $D^n$ is a 
standard $n$-ball and $S_-^{n-1}$ is the lower hemisphere of its boundary). 
Under this isomorphism $A$ corresponds to a subdivision of the upper hemisphere
$S_+^{n-1}$. Moreover, the pair $(D^n, S_+^{n-1})$ is PL isomorphic to 
$(S_+^{n-1}\times I, S_+^{n-1}\times 0)$.  Thus the 
pair $(A\times I,A\times 0)$ 
is PL isomorphic to a subdivision of the pair 
$(\sigma,A)$.  Part (g) of Definition \ref{defad} extends $F$ to 
$\Cell(A\times I)$,
and now part (f) of Definition \ref{defad} gives a corresponding extension of
$F$ to $\Cell(\sigma)$. 
\end{proof}

\section{The spectrum associated to an ad theory}
\label{spectrum}

\begin{definition}
\label{n15}
Let $\Delta_\inj$ denote the category whose objects are the sets
$\{0,\ldots,n\}$ and whose morphisms are the monotonically increasing 
injections.  By a {\it semisimplicial set} we mean a contravariant functor 
from $\Delta_\inj$ to Set.
\end{definition}

Thus a semisimplicial set is a simplicial set without degeneracies.  In the
literature these are often called $\Delta$-sets, but this seems awkward because
$\Delta$ is the category that governs simplicial sets.

The geometric realization of a semisimplicial set is defined by
\[
|A|=\bigl(\coprod \Delta^n \times A_n \bigr)/\mathord{\sim},
\]
where $\sim$ identifies $(d^iu,x)$ with $(u,d_ix)$.

\begin{definition}
Let $*$ denote the semisimplicial set with a single element (also denoted $*$) 
in each degree.  A {\it basepoint} for a semisimplicial set is a 
semisimplicial map from $*$.
\end{definition}

\begin{remark}
Geometric realization of semisimplicial sets is a left adjoint (for example by
\cite[Proposition 2.1]{MR0300281}), but it does not preserve quotients because
it does not take the terminal object $*$ to a point.
\end{remark}

Now fix an ad theory.  First we construct the spaces of the spectrum.

\begin{definition}
(i) For $k\geq 0$, let $P_k$ be the semisimplicial 
set with $n$-simplices
\[ 
(P_k)_n=\ad^k(\Delta^n)
\]
and the obvious face maps.  Give $P_k$ the basepoint determined by the
elements $\emptyset$.

(ii) Let $Q_k$ be $|P_k|$.
\end{definition}

Next we define the structure maps of the spectrum.  
For this we will use the semisimplicial analog of the Kan suspension.  

\begin{definition}
\label{n19}
Given a based semisimplicial set $A$, define $\Sigma A$ to be the based
semisimplicial set for which the only 0-simplex is $*$ and the (based) set of 
$n$ simplices for $n\geq 1$ is $A_{n-1}$. 
The face operators $d_i:(\Sigma A)_n\to (\Sigma A)_{n-1}$ agree with those of
$A$ for $i<n$ and $d_n$ takes all simplices to $*$.
\end{definition}

\begin{remark}
The motivation for this construction is that the cone on a simplex is a simplex
of one dimension higher.
\end{remark}

\begin{lemma}
\label{n17}
There is a natural homeomorphism
$\Sigma|A| \cong |\Sigma A|$.
\end{lemma}

\begin{proof}
If $t\in [0,1]$ and $u\in \Delta^{n-1}$ let us write $\langle t,u\rangle$ for 
the point $((1-t)u,t)$ of $\Delta^n$.
The homeomorphism takes $[t,[u,x]]$ (where $[\ ]$ denotes equivalence class) to
$[\langle t,u\rangle,x]$.
\end{proof}

Next observe that for each $n$ there is an isomorphism of $\Z$-graded
categories
\[
\theta: 
\Cell(\Delta^{n+1},\partial_{n+1}\Delta^{n+1}\cup \{n+1\})
\to
\Cell(\Delta^n) 
\]
which lowers degrees by 1,
defined as follows: a simplex $\sigma$ of $\Delta^{n+1}$ which is not in
$\partial_{n+1}\Delta^{n+1}\cup \{n+1\}$ contains the vertex $n+1$.  Let 
$\theta$ take $\sigma$ (with its canonical orientation) to the simplex of 
$\Delta^n$ spanned by the vertices of $\sigma$ other than $n+1$ (with 
$(-1)^{\dim\sigma-1}$ times its canonical orientation).  $\theta$ is
incidence-compatible (this is the reason for the sign in its definition) so
by part (e) of Definition
\ref{defad} it induces a bijection
\[
\theta^*:
\ad^k(\Delta^n)
\to
\ad^{k+1}(\Delta^{n+1},\partial_{n+1}\Delta^{n+1}\cup \{n+1\}).
\]

The composites
\[
\ad^k(\Delta^n)
\labarrow{\theta^*}
\ad^{k+1}(\Delta^{n+1},\partial_{n+1}\Delta^{n+1}\cup \{n+1\})
\to
\ad^{k+1}(\Delta^{n+1})
\]
give a semisimplicial map
\[
\Sigma P_k\to P_{k+1}.
\]

\begin{definition}
Let $\bQ$ be the spectrum consisting of the spaces $Q_k$ with the structure maps
\[
\Sigma Q_k=\Sigma|P_k|\cong |\Sigma P_k|\to |P_{k+1}|=Q_{k+1}.
\]
\end{definition}

In the rest of this section we show:

\begin{prop}
\label{n5}
$\bQ$ is an $\Omega$ spectrum.
\end{prop}

First we observe that the semisimplicial Kan suspension $\Sigma$ has a right
adjoint:

\begin{definition}
For a based semisimplicial set $A$ define a semisimplicial set $\Omega A$ 
by letting the $n$-simplices of $\Omega A$ be the $(n+1)$-simplices $x$ of
$A$ which satisfy the conditions
\[
d_{n+1} x=*\quad \text{and}\quad  (d_0)^{n+1}x=*.
\]
The face maps are induced by those of $A$.
\end{definition}

It's easy to check that the adjoint of the map $\Sigma P_k\to P_{k+1}$ is an 
isomorphism
\[
P_k\cong \Omega P_{k+1};
\]
it therefore suffices to relate the semisimplicial $\Omega$ to the usual one.

Recall (\cite[page 329]{MR0300281}) that a semisimplicial set $A$ is a {\it Kan
complex} if every map $\Lambda_{n,i}\to A$ (where
$\Lambda_{n,i}$ is defined on page 323 of \cite{MR0300281})
extends to a map $\Delta^n\to A$.
Proposition \ref{n5} 
follows from the next two facts.

\begin{lemma}
\label{n6}
If $A$ is a Kan complex then
the adjoint of the composite
\[
\Sigma|\Omega A|\cong |\Sigma \Omega A|\to |A|
\]
is a weak equivalence. 
\end{lemma}

\begin{lemma}
\label{n3}
For each $k$, $P_k$ is a Kan complex.
\end{lemma}

\begin{proof}[Proof of Lemma \ref{n6}.]
Let $S^n$ denote the based semisimplicial set with one nontrivial simplex in
degree $n$.  For a based Kan complex $B$, Remark 6.5 of \cite{MR0300281} gives
a bijection
\[
\pi_n(|B|)\cong [S^n,B]
\]
where $[\ ,\ ]$ denotes based homotopy classes of based semisimplicial maps 
(the homotopy relation is defined at the beginning of \cite[Section 
6]{MR0300281}).  It is easy to check that $\Omega A$ is a Kan complex if $A$ 
is. It therefore suffices to show that the adjunction induces a map
\[
[S^n,\Omega A]\to [\Sigma S^n, A]
\]
and that this map is a bijection.

For this, we first observe that for a based semisimplicial set $B$ the set of
based semisimplicial maps $S^n\to B$ can be identified with the set (which 
will be denoted by $\rho_n(B)$) of $n$-simplices of $B$ with all faces at the
basepoint.  Moreover, if $B$ is Kan then (by lines $-10$ to $-7$ of page 333 of
\cite{MR0300281}) the set $[S^n,B]$ is the quotient of $\rho_n(B)$ by the
relation which identifies $y$ and $y'$ if there is a $z$ with $d_0z=y$,
$d_1z=y'$, and $d_iz=*$ for $i>1$.  The desired bijection is immediate from
this and the fact that $\Sigma S^n$ is $S^{n+1}$.
\end{proof}

For the proof of Lemma \ref{n3} we need to introduce a useful class
of semisimplicial sets.

\begin{definition}
\label{n2}
A semisimplicial set is $\it strict$ if two simplices are equal whenever they
have the same set of vertices.
\end{definition}

Note that a strict semisimplicial set is the same thing as an ordered simplicial
complex. 

The geometric realization of a strict semisimplicial set $A$ has a 
canonical ball complex structure (which will also be denoted by $A$) and 
the cells have canonical orientations.

\begin{remark}
\label{n1}
We will make important use of the following observation 
(\cite[page 140]{MR1211640}):
for a pair $(A,B)$ of strict semisimplicial sets, there is a canonical 
bijection between the set of semisimplicial maps $(A,B)\to (P_k,*)$ and 
$\ad^k(A,B)$.
\end{remark}

\begin{proof}[Proof of Lemma \ref{n3}.]
By Remark \ref{n1}, it suffices to show that every element of
$\ad^k(\Lambda_{n,i})$ extends to an element of $\ad^k(\Delta^n)$, and this is
true by Lemma \ref{Dec8}.
\end{proof}

\section{$\bQ$ represents $T^*$}
\label{f11}

In this section we prove:

\begin{thm} \label{Theorem 5.2} The cohomology theory represented by $\bQ$ is 
naturally isomorphic to $T^*$.
\end{thm}

\begin{remark}
\label{n4}
Theorem  \ref{Theorem 5.2} includes as a special case  the statement that the 
semisimplicial sets ${\mathbb L}_n(\Lambda^*(K))$ and 
${\mathbb H}^n(K;{\mathbb L}_\bullet(\Lambda))$ in Proposition 13.7 of 
\cite{MR1211640} are weakly equivalent; the statement given in 
\cite{MR1211640} that they are actually isomorphic is not correct (because 
the sets in the 8th and 9th line of the proof are not isomorphic).
\end{remark}

Let $\S$ denote the category of pairs of finite strict semisimplicial sets
(see Definition \ref{n2}) and semisimplicial maps.  Let $\H$ be the homotopy 
category of finite CW pairs and let $R:\S\to \H$ be geometric realization.  A 
map $(f,g)$ in $\S$ is a {\it weak equivalence} if $(Rf,Rg)$ is a weak 
equivalence in $\H$.  Let $w^{-1}\S$ be the category obtained from $\S$ by 
inverting the weak equivalences.  

\begin{lemma}
$R$ induces an equivalence of categories 
\[
w^{-1}\S\to \H
\]
\end{lemma}

\begin{proof}
Let $\S'$ be the category of pairs of finite semisimplicial sets and 
semisimplicial maps, with weak equivalences defined by geometric realization, 
and let $w^{-1}\S'$ be the category obtained by inverting the weak equivalences.
Geometric realization induces an equivalence
\[
w^{-1}\S'\to \H
\]
by \cite[Theorem I.4.3 and Remark I.4.4]{MR0413113}.  Moreover, 
the map $w^{-1}\S\to w^{-1}\S'$ is an equivalence because 
every object of $\S'$ is weakly equivalent to an object of $\S$ (see
\cite[Proof of Theorem I.4.1]{MR0413113}; note that the second derived
subdivision of a semisimplicial set is a strict semisimplicial set).  
\end{proof}

Theorem \ref{Theorem 5.2} follows from the lemma and

\begin{prop}
\label{n12}
There is a 
natural transformation 
\[
\Xi: \bQ^*(|A|,|B|)\to T^*(A,B)
\]
of functors on $\S$ with the following properties:

(i) $\Xi$ is a bijection when $A=*$ and $B$ is empty.

(ii) The diagram
\[
\xymatrix{
Q^k(|B|)
\ar[r]^\Xi
\ar[d]
&
T^k(B)
\ar[d]
\\
Q^{k+1}(|A|,|B|)
\ar[r]^\Xi
&
T^{k+1}(A,B)
}
\]
commutes, where the vertical arrows are the connecting homomorphisms.

(iii) $\Xi$ is a homomorphism.
\end{prop}

The remainder of the section is devoted to the proof of Proposition \ref{n12}.
We begin with the construction of $\Xi$.

Recall that $T^*(A,B)$ is $\ad^k(A,B)$ modulo the equivalence relation 
$\sim$ defined by $F\sim G$ if and only if there is an $H\in 
\ad^k(A\times I,B\times I)$ which restricts to $F$ and $G$ on 
$A\times 0$ and $A\times 1$. (We remind the reader that we are using the same
symbol for a strict semisimplicial set and the ball complex it determines;
thus a symbol such as $A\times I$ denotes a product of ball complexes).

There is a similar description of $\bQ^*(|A|,|B|)$. By Proposition \ref{n3} and 
\cite[Remark 6.5]{MR0300281}, $\bQ^k(|A|,|B|)$ is the set $[(A,B),(P_k,*)]$ 
of homotopy classes of semisimplicial maps.  The homotopy relation for 
semisimplicial maps is defined at the beginning of Section 6 of 
\cite{MR0300281}; it uses the ``geometric product'' $\otimes$ defined in 
\cite[Section 3]{MR0300281}.  Using Remark \ref{n1} above, we see that 
$\bQ^k(|A|,|B|)$ is $\ad^k(A,B)$ modulo the equivalence relation $\sim'$ 
defined by: $F\sim' G$ if and only if there is an $H\in 
\ad^k(A\otimes I,B\otimes I)$ which restricts to $F$ on $A\otimes 0$ 
and to $G$ on $A\otimes 1$. 

We can now define $\Xi$: given an element $x\in\bQ^k(|A|,|B|)$, 
choose an element $F$ of $\ad^k(A,B)$ which represents it and let
$\Xi(x)$ be the class of $F$.
To see that this is well-defined, note that
$A\otimes I$ is a subdivision of $A\times I$, so by the gluing property
of ad theories we see that $F\sim' G$ implies $F\sim G$.  

\begin{remark}
The definition of $\Xi$ was suggested by the argument on page 140 of 
\cite{MR1211640}.
\end{remark}

The definition of $\otimes$ shows that 
$\Xi$ is the identity map when $A=*$ and $B$ is empty.

Next we check that $\Xi$ is natural.  It is obviously natural for
inclusions of pairs. If $(f,g):(A,B)\to (A',B')$ is any 
semisimplicial map, let $M_f$ and $M_g$ be the mapping cylinders
as defined on page 327 of \cite{MR0300281}; these are 
strict semisimplicial sets and have the property that there is an 
inclusion
\[
(i,j):(A,B)\to (M_f,M_g),
\]
an inclusion 
\[
(i',j'):(A',B')\to (M_f,M_g)
\]
which is a weak equivalence,
and a homotopy
$|(i',j')|\circ |(f,g)|\simeq |(i,j)|$.
Then $|(f,g)|^*:Q^*(|A'|,|B'|)\to Q^*(|A|,|B|)$ is equal to 
$(|(i',j')|^*)^{-1}|(i,j)|^*$, and
similarly for $(f,g)^*:T^*(A',B')\to T^*(A,B)$.  Hence
$(f,g)^*\circ\Xi=\Xi\circ|(f,g)|^*$.

For the proof of part (ii) of Proposition \ref{n12} we need the Kan cone
construction (because the Kan suspension of a strict semisimplicial set is 
not strict in general). 

\begin{definition}
Let $A$ be a semisimplicial set.  Define a semisimplicial set 
$CA$ as follows.  The 0-simplices of $CA$ are the 0-simplices 
of $A$ together with a 0-simplex $c$.
For $n\geq 1$ the $n$ simplices of $CA$ are $A_n\coprod 
A_{n-1}$.  If the inclusions of $A_n$ and $A_{n-1}$ in $(CA)_n$ are denoted 
by $f$ and $g$ then
the face maps $d_i:(CA)_n\to(CA)_{n-1}$ are defined by
\[
d_i f(x)=f(d_ix)
\]
for all $i$ and
\[
d_i g(x)=
\begin{cases}
c & \text{if $n=1$ and $i=0$} \\
g(d_i x) & \text{if $n>1$ and $i<n$} \\
f(x) & \text{if $i=n$}.
\end{cases}
\]
\end{definition}

We leave it to the reader to check that $|CA|\cong C|A|$,
where $C|A|$ denotes $I\wedge (|A|_+)$ (we choose the basepoint of $I$ to be 1).
Note that there is an inclusion $A\to CA$ and that the quotient 
$CA/(A\cup c)$ is $\Sigma (A_+)$ (where $A_+$ denotes $A$ with a disjoint
basepoint).

If $A$ is strict then $CA$ is also.

\begin{proof}[Proof of Proposition \ref{n12}(ii).]
The unreduced suspension isomorphism
\[
Q^k(|B|)\to Q^{k+1}(C|B|, |B|\cup |c|)
\]
is defined as follows: given $f:|B|\to Q_k$ the composite 
\[
C|B|\labarrow{Cf} CQ_k\to \Sigma Q_k \to Q_{k+1}
\]
takes $|B|\cup |c|$ to the basepoint, and therefore represents an element of
$Q^{k+1}(C|B|, |B|\cup |c|)$.

There is an isomorphism of categories
\[
\mu: \Cell(CB,B\cup c)
\to
\Cell(I\times B, \{0,1\}\times B)
\]
defined as follows: a simplex
$\sigma$ of $CB$ which is not in $B\cup c$ corresponds to a simplex 
$\sigma'$ of $B$; let $\mu$ take $\sigma$ (with its canonical orientation) 
to $I\times \sigma'$ (with $(-1)^{\dim\sigma'}$ times its canonical 
orientation).

There is a similar isomorphism
\[
\nu:\Cell(CA,B\cup c)
\to
\Cell(I\times A, (1\times A)\cup (0\times B)).
\]

Both $\mu$ and $\nu$ are incidence-compatible (Definition \ref{n25}(i)) so part
(e) of Definition \ref{defad} applies.

It is easy to check that the diagram
\[
\xymatrix{
Q^k(|B|)
\ar[rr]^-\Xi
\ar[d]_\cong
&&
T^k(B)
\ar[d]^{\kappa^*}
\\
Q^{k+1}(|CB|,|B\cup c|)
\ar[r]^-\Xi
&
T^{k+1}(CB,B\cup c)
\ar[r]^-{\mu^*}
&
T^{k+1}(I\times B,\{0,1\}\times B)
\\
Q^{k+1}(|CA|,|B\cup c|)
\ar[u]^\cong
\ar[d]
\ar[r]^-\Xi
&
T^{k+1}(CA,B\cup c)
\ar[u]^\cong
\ar[r]^-{\nu^*}
&
T^{k+1}(I\times A,(1\times A)\cup (0\times B))
\ar[u]^\cong
\ar[d]
\\
Q^{k+1}(|A|,|B|)
\ar[rr]^-\Xi
&&
T^{k+1}(A,B)
}
\]
commutes.  The vertical composite on the right is by definition the negative 
of the connecting homomorphism, so it suffices to show the same for the 
vertical composite on the left.

The connecting homomorphism 
\[
Q^k(|B|)\to Q^{k+1}(|A|,|B|)
\]
is defined to be the composite 
\begin{multline*}
Q^k(|B|)
\stackrel{\cong}{\to}
Q^{k+1}(C|B|,|B|\cup |c|)
\cong
\tilde{Q}^{k+1}(C|B|/(|B|\cup |c|))
\\
\to
\tilde{Q}^{k+1}(|A|\cup C|B|)
\stackrel{\cong}{\leftarrow}
\tilde{Q}^{k+1}(|A|/|B|)
\end{multline*}
where the third and fourth maps are induced by the evident quotient maps.

It now suffices to note that the diagram
\[
\xymatrix{
C|B|/(|B|\cup c)
\ar[rd]^i
&
\\
|A|\cup C|B|
\ar[u]
\ar[d]
&
C|A|/(|B|\cup c)
\\
|A|/|B|
\ar[ru]^j
}
\]
homotopy commutes, where $i$ takes $t\wedge b$ to the class of $(1-t)\wedge 
b$ and $j$ takes the class of $a$ to the class of $0\wedge a$. (The homotopy 
is given by $h(a,s)=s\wedge a$ for $a\in |A|$ and $h(t\wedge b)=s(1-t)\wedge
b$.) The negative sign mentioned above comes from the $1-t$ in the definition of
$i$.
\end{proof}

It remains to prove part (iii) of Proposition \ref{n12}.

First recall that for any cohomology theory $E$ the addition in 
$E^k(|A|,|B|)$ is the
composite
\begin{multline*}
E^k(|A|,|B|) \times E^k(|A|,|B|)
=
\tilde{E}^k(|A|/|B|) \times \tilde{E}^k(|A|/|B|)
\\
\cong
\tilde{E}^{k+1}(\Sigma(|A|/|B|))
\times
\tilde{E}^{k+1}(\Sigma(|A|/|B|))
\stackrel{\cong}{\leftarrow}
\tilde{E}^{k+1}(\Sigma(|A|/|B|)\vee\Sigma(|A|/|B|))
\\
\stackrel{p^*}{\to}
\tilde{E}^{k+1}(\Sigma(|A|/|B|))
\cong
\tilde{E}^k(|A|/|B|)
=
E^k(|A|,|B|)
\end{multline*}
where $p$ is the pinch map.

It therefore suffices to observe that, by part (ii) and naturality, the diagram
\[
\xymatrix{
\tilde{Q}^k(|A|/|B|)
\ar[r]^\Xi
\ar[d]_\cong
&
\tilde{T}^k(|A|/|B|)
\ar[d]^\cong
\\
\tilde{Q}^{k+1}(\Sigma(|A|/|B|))
\ar[r]^\Xi
&
\tilde{T}^{k+1}(\Sigma(|A|/|B|))
}
\]
commutes, where the vertical arrows are the suspension
isomorphisms of the reduced cohomology theories $\tilde{Q}^*$ and $\tilde{T}^*$.

\section{The symmetric spectrum associated to an ad theory}
\label{f12}

Symmetric spectra were originally defined simplicially (\cite[Definition
1.2.1]{MR1695653}).  The topological definition is the obvious analog
(\cite[Example 4.2]{MR1806878}):

\begin{definition}
A symmetric spectrum $\bX$ consists of 

(i)
a sequence $X_0,X_1,\ldots$ of pointed topological spaces, 

(ii)
a pointed map $s:S^1\wedge X_k\to X_{1+k}$ for each $k\geq 0$, and 

(iii)
a based left $\Sigma_k$-action on $X_k$, 

such that the composition
\[
S^p\wedge X_k
\labarrow{S^{p-1}\wedge s}
S^{p-1}\wedge X_{1+k}
\to
\cdots
\to
X_{p+k}
\]
is $\Sigma_p\times \Sigma_k$-equivariant for each $p\geq 1$ and $k\geq 0$.
\end{definition}

Our first goal in this section is to define a symmetric spectrum associated to
an ad theory.  In order to have a suitable $\Sigma_k$ action we will construct
the $k$-th space of the spectrum as the geometric realization of a $k$-fold 
multisemisimplicial set; the $\Sigma_k$ action will come from permutation of 
the semisimplicial directions.

By a $k$-fold multisemisimplicial set we mean a functor from 
$\Delta_\inj^k$ to sets (see Definition \ref{n15}).
Given a multiindex $\bfn=(n_1,\ldots,n_k)$, let
$\Delta^\bfn$ denote the product
\[
\Delta^{n_1}\times\cdots\times\Delta^{n_k}.
\]
The geometric realization of a $k$-fold multisemisimplicial set $A$ is 
\[
|A|=\bigl(\coprod \Delta^\bfn \times A_\bfn \bigr)/\mathord{\sim},
\]
where $\sim$ denotes the evident equivalence relation.

Now fix an ad theory. 

\begin{definition}\label{9.1} For each $k\geq 1$, define a 
$k$-fold multisemisimplicial
set $R_k$ by 
\[
(R_k)_{\bfn}=\ad^k(\Delta^\bfn).
\]
Let $M_k$ be the geometric realization of $R_k$.
For $k=0$, let $R_0$ be the set of $*$-ads of degree 0 and
let $M_0$ be $R_0$ with the discrete topology.
\end{definition}


Our next definition gives the left action of $\Sigma_k$ on $M_k$.  An element 
of $M_k$ has the form $[u,F]$, where $u=(u_1,\ldots,u_k)\in \Delta^\bfn$, $F\in
\ad^k(\Delta^\bfn$), and $[\ ]$ denotes equivalence class.  Given 
$\eta\in\Sigma_k$ let $\epsilon(\eta)$ denote 0 if $\eta$ is even and 1 if 
$\eta$ is odd.  

\begin{definition}
\label{9.2}
Define
\[
\eta([u,F])=
[(u_{\eta^{-1}(1)},\ldots,u_{\eta^{-1}(k)}),i^{\epsilon(\eta)} \circ F\circ 
\eta_\#].
\]
Here $i$ is the involution in the target category of the ad theory 
and 
$\eta_\#$ is the map
\[
\Cell(\Delta^{n_{\eta^{-1}(1)}}\times\cdots\times\Delta^{n_{\eta^{-1}(k)}})
\to
\Cell(\Delta^{n_1}\times\cdots\times\Delta^{n_k})
\]
which takes 
\[
(\sigma_{\eta^{-1}(1)}\times\cdots\times \sigma_{\eta^{-1}(k)},
o_{\eta^{-1}(1)}\times\cdots\times o_{\eta^{-1}(k)})
\]
to
\[
(\sigma_1\times\cdots\times \sigma_k, o_1\times\cdots\times o_k).
\]
\end{definition}

It remains to define the suspension maps.

\begin{definition}\label{12.4}
{\rm (i)}
For each ball complex $K$ let
\[
\lambda:\Cell(\Delta^1 \times K,\partial \Delta^1\times K)
\to
\Cell(K)
\]
be the incidence-compatible isomorphism of categories which takes 
$\Delta^1\times(\sigma,o)$ (where $\Delta^1$ is given its standard 
orientation) to $(\sigma,o)$.

{\rm (ii)}
Given $t\in[0,1]$ let $\bar{t}$ denote the point
$(1-t,t)$ of $\Delta^1$.

{\rm (iii)}
Given $k\geq 1$ let 
\[
s:S^1\wedge M_k\to M_{1+k}
\]
be the map which takes $[t,[u,F]]$ to
$[(\bar{t},u),\lambda^*(F)]$.  
\end{definition}

\begin{prop} 
The sequence $M_0,M_1,\ldots$, with the $\Sigma_k$-actions given by Definition
\ref{9.2} and the suspension maps given by Definition \ref{12.4}(iii), is a
symmetric spectrum.
\qed
\end{prop}

We will denote this symmetric spectrum by $\bM$.

\begin{example}
\label{Nov12}
Let us write $\bM_{\pi,Z,w}$ (resp., $\bM^R$) for the symmetric spectrum
associated to $\ad_{\pi,Z,w}$ (resp., $\ad^R$).  The morphism of ad theories
\[
\Sig:
\ad_{\pi,Z,w}\to\ad_{\Z[\pi]^w}
\]
(see Proposition \ref{j9})
induces a map 
\[
\bM_{\pi,Z,w}\to \bM^{\Z[\pi]^w}.
\]
\end{example}

In the remainder of this section we show that $\bM$ is weakly equivalent (in an
appropriate sense) to the spectrum $\bQ$ defined in Section \ref{spectrum}.

For $k\geq 1$, let $Q'_k$ be the realization of the semisimplicial set 
with $n$-simplexes
\[
(R_k)_{(0,\ldots,0,n)}
\]
Then $Q'_k$ is homeomorphic to $Q_k$, and there is an obvious map $Q'_k\to 
M_k$, so we get a map
\[
Q_k\to M_k
\]
for $k\geq 1$.

\begin{prop}\label{Proposition 6.2}
The map $Q_k\to M_k$ is a weak equivalence.
\end{prop}

\begin{prop}
\label{n16}
The diagram
\[
\xymatrix{
\Sigma Q_k
\ar[r]
\ar[d]
&
\Sigma M_k
\ar[d]
\\
Q_{1+k}
\ar[r]
&
M_{1+k}
}
\]
commutes up to homotopy.
\end{prop}

Before proving these we deduce some consequences.  As in \cite{MR1806878}, let 
the forgetful functor from symmetric spectra to ordinary spectra (which are 
called prespectra in \cite{MR1806878}) be denoted by $\mathbb U$.  It is 
shown in \cite{MR1806878} that the right derived functor $R{\mathbb U}$ is an 
equivalence of homotopy categories.

\begin{cor} 
{\rm (i)}
 $\bM$ is a positive $\Omega$ spectrum (that is, the map
$M_k\to \Omega M_{1+k}$ is a weak equivalence for $k\geq 1$).

{\rm (ii)}
$R{\mathbb U}$ takes $\bM$ to $\bQ$.

{\rm (iii)}
The homotopy groups of $\bM$ are the bordism groups of the ad theory.
\end{cor}

\begin{proof}  Part (i) is immediate from the proposition. 

For part (ii), first recall that $(R{\mathbb U})\bM$ is defined to be 
$\mathbb U$
of a fibrant replacement of $\bM$.  But by \cite[Example 4.2]{MR2421129}
$\bM$ is semistable, which means that the map from $\bM$ to its fibrant 
replacement is a
$\pi_*$-isomorphism.  It follows that $(R{\mathbb U})\bM$ is (up to weak
equivalence) ${\mathbb U}\bM$, and it therefore suffices to show that $\bQ$ is 
weakly equivalent to ${\mathbb U}\bM$.  
Define a spectrum $\bX$ as follows: 
$X_0$
is $*$, $X_1$ is $Q_1$, and and for $k\geq 2$ $X_k$ is
the iterated mapping cylinder of the sequence of maps
\[
\Sigma^{k-1}Q_1 
\labarrow{\Sigma^{k-2}s}
\Sigma^{k-2}Q_{2} 
\labarrow{\Sigma^{k-3}s}
\cdots
\Sigma Q_{k-1}
\labarrow{s}
Q_k
\]
The maps $\Sigma X_k\to X_{1+k}$ are defined to be the obvious inclusion maps.
Then there are evident weak equivalences $\bX\to \bQ$ and (using Propositions
\ref{Proposition 6.2} and \ref{n16}) $\bX\to \bM$, which proves part (ii).

Part (iii) is immediate from part (ii).
\end{proof}

For the proof of Proposition \ref{Proposition 6.2} we will
use an idea adumbrated on page 695 of \cite{MR1694381}.

First we interpolate between $Q_k$ and $M_k$.  For $1\leq m\leq k$ let 
$R_k^m$ be the $m$-fold multisemisimplicial set defined by 
\[
(R_k^m)_\bfn=(R_k)_{0,\ldots,0,\bfn}.
\]
We have $|R_k^1|=Q_k$ and $|R_k^k|=M_k$, so it suffices to show that the 
inclusion $|R_k^{m-1}|\to|R_k^m|$ is a weak equivalence for each $m\geq 2$.
We will prove this for each $k$ by induction on $m$, so we
assume
\[
\tag{*}
\text{$|R_k^{m'-1}|\to|R_k^{m'}|$ is a weak equivalence if $m'<m$.}
\]

Next we observe that the realization $|R_k^m|$ can be obtained by first 
realizing in the last $m-1$ semisimplicial directions and then realizing in 
the remaining direction.  Namely, for each $p\geq 0$ let $R_k^m[p]$ be the 
$(m-1)$-fold semisimplicial set with 
\[
(R_k^m[p])_\bfn=
(R_k)_{0,\ldots,0,p,\bfn}.
\]
As $p$ varies we obtain a semisimplicial space $|R_k^m[\bu]|$ whose 
realization is $|R_k^m|$.  Now $R_k^{m-1}$ is $R_k^m[0]$, and the inclusion 
$|R_k^{m-1}|\to|R_k^m|$ is the inclusion of the space of 0-simplices 
$|R_k^m[0]|$ in $|R_k^m|$.  It therefore suffices to show that the latter map
is a weak equivalence, and this is part (v) of:

\begin{lemma}\label{Lemma 6.4}  
{\rm (i)}
In the semisimplicial space $|R_k^m[\bu]|$, all face maps
are homotopy equivalences. 

{\rm (ii)}
For each $p$, all of the face maps from $|R_k^m[p]|$ to $|R_k^m[p-1]|$ 
are homotopic.

{\rm (iii)}
The map $|R_k^m[0]|\to |R_k^m|$ is a homology isomorphism.

{\rm (iv)}
The map $|R_k^m[0]|\to |R_k^m|$ is (up to weak equivalence) an $H$-map 
between grouplike $H$-spaces.

{\rm (v)}
The map $|R_k^m[0]|\to |R_k^m|$ is a weak equivalence.
\end{lemma}

For the proof of the lemma we need an auxiliary construction.   Let ad denote
the ad theory we have fixed and let $\A$ be its target category.  Given a ball
complex $L$ we can define a new $\Z$-graded category $\A [L]$ by letting the
set of objects in dimension $n$ be 
$\pre^{-n}(L)$.  Now we define an ad theory $\ad [L]$ with values in $\A [L]$ 
by letting
$\ad [L]^j(K)$ consist of the pre-$K$-ads which correspond to
$(K\times L)$-ads under the bijection
\[
\pre [L]^j (K) \cong
\pre^j(K\times L).
\]

Let us write 
$\bQ [L]$ and $R [L]^m_k$ for the spectrum
and the multisemisimplicial sets constructed from the theory $\ad [L]$.

\begin{proof}[Proof of Lemma \ref{Lemma 6.4}] 
Part (i). First note that 
$R_k^m[p]$ is the same thing as
$R[\Delta^p]^{m-1}_k$.  By the inductive hypothesis (*) we know that 
$|R[\Delta^p]^{m-1}_k|$ is weakly equivalent to $Q [\Delta^p]_k$.  
The homotopy groups of $Q [\Delta^p]_k$ are (up to a shift in dimension)
the bordism groups of the ad theory $\ad [\Delta^p]$, and
inspection of the definitions shows that these
are the groups $T^{-*}(\Delta^p)$.
This implies that all face maps 
in $R_k^m[\bu]$
are weak equivalences, and hence homotopy equivalences since
all spaces are CW complexes.

Part (ii) follows from part (i) and the semisimplicial identities.

Part (iii) follows from part (ii) and the homology spectral sequence of a
semisimplicial space (cf.\ \cite[Theorem 11.4]{MR0420610}),
but note that our situation
is simpler because there are no degeneracy maps.

Part (iv): 
Let $Q[\Delta^\bu]_k$ denote the semisimplicial space which is
$Q[\Delta^p]_k$ in degree $p$.
By the inductive hypothesis (*), it suffices to show that the map
\[
Q_k=Q[\Delta^0]_k \to |Q[\Delta^\bu]_k|
\]
is (up to weak equivalence) an $H$-map between grouplike $H$-spaces, and this
is a consequence of the following commutative diagram 
(where $\hat{Q}_{k+1}$ denotes the basepoint component of $Q_{k+1}$; note 
that $\Omega\hat{Q}_{k+1}$ is the same thing as $\Omega Q_{k+1}$):
\[
\xymatrix{
Q_k
\ar[d]_{\simeq}
\ar[r]
&
|Q[\Delta^\bu]_k|
\ar[d]_{\simeq}^{\alpha}
\\
\Omega\hat{Q}_{k+1}
\ar[r]
\ar[rd]
&
|\Omega\hat{Q}[\Delta^\bu]_{k+1}|
\ar[d]_{\simeq}^\beta
\\
&
\Omega|\hat{Q}[\Delta^\bu]_{k+1}|
}
\]
Here $\alpha$ is a weak equivalence by 
\cite[Theorem A.4(ii)]{MR0339152},
and 
$\beta$ is a weak equivalence by 
\cite[Theorem 12.3]{MR0420610} (this is where
we need to use basepoint-components).

Part (v) now follows from parts (iii) and (iv) and 
\cite[Corollary IV.3.6 and Corollary IV.7.9]{MR516508}.
\end{proof}

We now turn to the proof of Proposition \ref{n16}.  For simplicity
we will do the case $k=2$; the general case is exactly the same but the 
notation is a little more complicated.

First let us give an explicit description of the maps in the diagram.  

Given an element $F\in \ad^2(\Delta^n)$, let us write $F'$ for the
corresponding element of $\ad^2(\Delta^0\times \Delta^n)$.  Then
the map $Q_2\to M_2$ takes $[u,F]$ to $[(1,u),F']$ (where 1 denotes the 
unique element of $\Delta^0$).  

Hence the clockwise composite in the diagram of 
\ref{n16} takes an element $[t,[u,F]]$ of $\Sigma Q_2$ to 
$[(\bar{t},1,u),\lambda^*(F')]$ (see Definition \ref{12.4}).

To describe the counterclockwise composite
we need some notation.  Recall that the homeomorphism in Lemma \ref{n17} takes
$[t,[u,x]]$ to $[\langle t,u\rangle,x]$, where 
$\langle t,u\rangle=((1-t)u,t)$.
Also, recall the
isomorphism
\[
\theta: 
\Cell(\Delta^{n+1}, \partial_{n+1}\Delta^{n+1} \cup \{n+1\})
\to
\Cell(\Delta^n)
\]
defined after Lemma \ref{n17}.

The map $\Sigma Q_2\to Q_3$ takes $[t,[u,F]]$ to 
$[\langle t,u\rangle, \theta^*F]$, and thus the counterclockwise 
composite in
the diagram of \ref{n16} takes $[t,[u,F]]$ to 
$[(1,1,\langle t,u\rangle,((\theta^* F)')']$.

Now we need a lemma:

\begin{lemma} 
For every $n\geq 0$ there is an incidence-compatible isomorphism
\begin{multline*}
\mu_n:
\Cell(\Delta^1\times \Delta^0\times \Delta^{n+1},
(\{1\}\times \Delta^0\times \Delta^{n+1})
\cup
(\Delta^1\times\Delta^0\times \{n+1\})
\\
\cup
(\{0\}\times\Delta^0\times\partial_{n+1}\Delta^{n+1}))
\to
\Cell(\Delta^n\times I)
\end{multline*}
(which lowers degrees by 1) such that

{\rm (a)}
$\mu_n$ takes the cell $\Delta^1\times\Delta^0\times \Delta^{n+1}$ (with
its standard orientation) to the cell $\Delta^n\times I$ (with 
its standard orientation).

{\rm (b)}
$\mu_n$ restricts to a morphism
\[
\Cell(\{0\}\times\Delta^0\times \Delta^{n+1},
\{0\}\times \Delta^0\times(\partial_{n+1}\Delta^{n+1} \cup \{n+1\}))
\to
\Cell(\Delta^n\times \{0\})
\]
which agrees with $\theta$. 

{\rm (c)}
$\mu_n$ restricts to a morphism
\[
\Cell(\Delta^1\times\Delta^0\times\partial_{n+1}\Delta^{n+1},
\partial \Delta^1\times\Delta^0\times\partial_{n+1}\Delta^{n+1}
)
\to
\Cell(\Delta^n\times \{1\})
\]
which agrees with $\lambda$.

{\rm (d)}
for $0\leq i\leq n$, $\mu_n$ restricts to a morphism
\begin{multline*}
\Cell(\Delta^1\times \Delta^0\times \partial_i\Delta^{n+1},
(\{1\}\times \Delta^0\times \partial_i\Delta^{n+1})
\cup
(\Delta^1\times\Delta^0\times \{n+1\})
\\
\cup
(\{0\}\times\Delta^0\times\partial_i\partial_{n+1}\Delta^{n+1}))
)
\to
\Cell(\partial_i\Delta^n\times I)
\end{multline*}
which agrees with $i\circ\mu_{n-1}$.
\end{lemma}

The proof is an easy induction. 

%
%
%
%
%

Now we can write down the homotopy 
\[
H:(\Sigma Q_2)\times I
\to
M_3
\]
needed for the case $k=2$ of \ref{n16}:
\[
H([t,[u,F]],s)=
\begin{cases}
[(\bar{t},1,\langle 2ts,u\rangle),\mu_n^*(J(F))]
&\text{if $0\leq s\leq 1/2$},\\
[\overline{(2-2s)t},1,\langle t,u\rangle),\mu_n^*(J(F))]
&\text{if $1/2\leq s\leq 1$},
\end{cases}
\]
where $J$ is the cylinder (see Definition \ref{defad}(g)).
It's easy to check that this is well-defined and that it is equal to the 
clockwise composite in the diagram of \ref{n16} when $s=0$ and to 
the
counterclockwise composite when $s=1$.

\section{Multiplicative ad theories}
\label{mult}

\begin{definition} \label{Definition 7.2}Let $\A$ be a $\Z$-graded category.  
A {\it strict monoidal structure} on $\A$ is a strict monoidal structure 
$(\boxtimes, \varepsilon)$ (see \cite[Section VII.1]{MR1712872}) on the 
underlying category such that

(a) the monoidal product $\boxtimes$ adds dimensions and the dimension of the 
unit element $\varepsilon$ is $0$,

(b) $i(x\boxtimes y)=(ix)\boxtimes y=x\boxtimes (iy)$ for all objects $x$ and $y$,
and similarly for morphisms,

(c) $x\boxtimes \emptyset_n=\emptyset_n\boxtimes x=\emptyset_{n+\dim x}$ for 
all $n$ and all objects $x$, and if $f:x\to y$ is any morphism then
$f\boxtimes \emptyset_n$ and $\emptyset_n\boxtimes f$ are each equal to the
canonical map $\emptyset_{n+\dim x}\to \emptyset_{n+\dim y}$.
\end{definition}

The category $\A_\STop$ of Example \ref{STop} and the category $\A_\STopFun$ of
Section \ref{moregeom} are examples.  Another example is the category $\A_C$ 
of Example \ref{gamma} when $C$ is a DGA.

\begin{assumption}
\label{n23}
From now on we will assume that Cartesian products in Section \ref{geom}
and tensor products in Section \ref{sym} are strictly associative (that is, we
assume that the monoidal categories Set and Ab 
have been replaced in those sections by equivalent strict monoidal 
categories; see \cite[Section XI.5]{MR1321145}).  
\end{assumption}

With this assumption, the category $\A_{e,*,1}$ defined in Section \ref{geom} 
(see Remark \ref{j2} for the notation)
and, when $R$ is commutative, the
category $\A^R$ defined in Section \ref{sym} are strict monoidal $\Z$-graded 
categories.

\begin{remark}\label{7.3}
If $\A$ is a $\Z$-graded category with a strict monoidal
structure, there is a natural map
\[
\boxtimes:\pre^k(K)\times\pre^l(L)\to\pre^{k+l}(K\times L)
\]
defined by 
\[
(F\boxtimes G)(\sigma\times\tau,o_1\times o_2)
=
i^{l\dim (\sigma)}F(\sigma,o_1)\boxtimes G(\tau,o_2);
\]
this is well-defined, because 
\[ 
F(\sigma,-o_1)\boxtimes
G(\tau,-o_2)=iF(\sigma,o_1)\boxtimes iG(\tau,o_2)=F(\sigma,o_1)\boxtimes
G(\tau,o_2).
\]
\end{remark}

\begin{definition} 
\label{n20}
A {\it multiplicative ad theory} is an ad theory together 
with a strict monoidal structure on the target category $\A$, such that

(a) the pre $*$-ad with value $\varepsilon$ is an ad, and

(b) the map in Remark \ref{7.3} restricts to a map
\[
\boxtimes:\ad^k(K)\times\ad^l(L)\to\ad^{k+l}(K\times L).
\]
\end{definition}

Examples are $\ad_C$ when $C$ is a DGA, $\ad_\STop$, $\ad_\STopFun$, 
$\ad_{e,*,1}$, and $\ad^R$ when $R$ is commutative; we will put the last two 
examples in a more general context in the next section.

\begin{thm} 
\label{n24}
The symmetric spectrum $\bM$ determined by a multiplicative ad theory is a
symmetric ring spectrum.  
\end{thm}

\begin{remark}
(i) Note that a symmetric ring spectrum satisfies {\it strict}
associativity, not just associativity up to homotopy.

(ii) The diagram
\[
\ad_\STop
\leftarrow
\ad_\STopFun
\rightarrow
\ad_{e,*,1}
\]
constructed in Section \ref{moregeom} gives a diagram
\[
\bM_\STop
\leftarrow
\bM_\STopFun
\rightarrow
\bM_{e,*,1}
\]
of symmetric ring spectra in which the first arrow is a weak equivalence.
\end{remark}

For the proof of Theorem \ref{n24} we need a lemma.  Recall Definitions
\ref{9.2} and \ref{12.4}.

\begin{lemma}
\label{n27}
Let $\bfn=(n_1,\ldots,n_k)$ and let $m\geq 0$.
Let $F\in\ad^k(\Delta^\bfn)$ and let $E$ be the $*$-ad with value 
$\varepsilon$.  Then

{\rm (i)}
$((\lambda^*)^m E)\boxtimes F=(\lambda^*)^m F$, and

{\rm (ii)}
$F \boxtimes ((\lambda^*)^m E)
=
i^{km} \circ (((\lambda^*)^m E)\boxtimes F)\circ\eta_\#$, 
where $\eta\in\Sigma_{k+m}$ is the permutation that moves the first $k$ 
elements to the end.
\end{lemma}

\begin{proof}
Let $(\sigma,o)$ be the 1-cell of $\Delta^1$ with its standard orientation and
let $(\tau,o')$ be an oriented cell of $\Delta^\bfn$ of dimension $l$.
For part (i), we have
\begin{align*}
(((\lambda^*)^m E)\boxtimes F)((\sigma,o)^{\times m}\times (\tau,o'))
&=
i^{km}(((\lambda^*)^m E)((\sigma,o)^{\times m})
\boxtimes
F(\tau,o'))
\\
&=
i^{km}(\varepsilon\boxtimes F(\tau,o'))
\\
&=
i^{km}F(\tau,o')
\end{align*}
and (using Definition \ref{n25}(ii))
\begin{align*}
((\lambda^*)^m F)((\sigma,o)^{\times m}\times (\tau,o'))
&=
(i^{km}\circ F\circ \lambda^m)((\sigma,o)^{\times m}\times (\tau,o'))
\\
&=
i^{km}F(\tau,o').
\end{align*}
For part (ii) we have
\begin{align*}
(F \boxtimes ((\lambda^*)^m E))((\tau,o')\times (\sigma,o)^{\times m})
&=
i^{lm}(F(\tau,o')\boxtimes ((\lambda^*)^m E)((\sigma,o)^{\times m}))
\\
&=
i^{lm}(F(\tau,o')\boxtimes\varepsilon)
\\
&=
i^{lm}F(\tau,o')
\end{align*}
and
\begin{align*}
(i^{km} \circ (((\lambda^*)^m E)\boxtimes F)\circ\eta_\#)
&((\tau,o')\times
(\sigma,o)^{\times m})
\\
&=
i^{km+lm}((((\lambda^*)^m E)\boxtimes F)((\sigma,o)^{\times m}\times (\tau,o')))
\\
&=
i^{lm}(((\lambda^*)^m E)((\sigma,o)^{\times m})\boxtimes F(\tau,o'))
\\
&=
i^{lm}(\varepsilon\boxtimes F(\tau,o'))
\\
&=
i^{lm}F(\tau,o').
\end{align*}
\end{proof}

\begin{proof}[Proof of Theorem \ref{n24}]
Recall (\cite[Definition 2.2.3]{MR1695653}) that the smash product 
$\bM\wedge\bM$ is defined to be the coequalizer of
\[
\xymatrix{
\bM\otimes\bS\otimes\bM
\ar@<0.5ex>[r]^-{1\otimes s}
\ar@<-0.5ex>[r]_-{r\otimes 1}
&
\bM\otimes\bM.
}
\]
Here $\otimes$ is the tensor product of the underlying symmetric sequences
(\cite[Definition 2.1.3]{MR1695653}), $\bS$ is the symmetric sphere spectrum
(\cite[Example 1.2.4]{MR1695653}), $s:\bS\otimes\bM\to \bM$ is induced by the
symmetric spectrum structure of $\bM$ (\cite[proof of Proposition
2.2.1]{MR1695653}), and $r$ is the composite
\[
\bM\otimes \bS\stackrel{t}{\to}
\bS\otimes\bM
\stackrel{s}{\to}
\bM,
\]
where $t$ is the twist isomorphism (\cite[page 160]{MR1695653}).

The $\boxtimes$ operation of Definition \ref{n20}(ii) gives an associative 
multiplication
\[
m:\bM\otimes\bM\to \bM
\]
and we need to show that this induces a map $\bM\wedge\bM\to \bM$.
Let
\[
\iota:\bS\to \bM
\]
be the map of symmetric spectra that takes the nontrivial element of $S^0$ to
the $*$-ad with value $\varepsilon$.
Lemma \ref{n27} shows that
the diagrams
\begin{equation}
\label{n21}
\xymatrix{
\bS\otimes \bM 
\ar[rr]^{\iota\otimes 1}
\ar[rd]_s
&&
\bM\otimes\bM
\ar[ld]^m
\\
&\bM&
}
\end{equation}
and
\begin{equation}
\label{n22}
\xymatrix{
\bM\otimes \bS
\ar[r]^{1\otimes \iota}
\ar[dd]_t
&
\bM\otimes \bM
\ar[rd]^m
\\
&&\bM
\\
\bS\otimes \bM
\ar[r]^{\iota\otimes 1}
&
\bM\otimes\bM
\ar[ru]^m
}
\end{equation}
commute.
These in turn imply that the diagram
\[
\xymatrix{
\bM\otimes \bS \otimes \bM
\ar[r]^-{1\otimes s}
\ar[dd]_{t\otimes 1}
&
\bM\otimes \bM
\ar[rd]^m
\\
&&\bM
\\
\bS\otimes \bM\otimes\bM
\ar[r]^-{s\otimes 1}
&
\bM\otimes\bM
\ar[ru]^m
}
\]
commutes, and hence $m$ induces an associative multiplication
\[
\bM\wedge\bM\to \bM.
\]
Moreover, diagrams \eqref{n21} and \eqref{n22} imply that the unit diagrams
\[
\xymatrix{
\bS\wedge \bM
\ar[rr]^{\iota\wedge 1}
\ar[rd]_\cong
&&
\bM\wedge\bM
\ar[ld]^m
\\
&\bM&
}
\]
and
\[
\xymatrix{
\bM\wedge \bS
\ar[rr]^{1\wedge\iota}
\ar[rd]_\cong
&&
\bM\wedge\bM
\ar[ld]^m
\\
&\bM&
}
\]
commute.  Thus $\bM$ is a symmetric ring spectrum. 
\end{proof}

\section{Geometric and symmetric Poincar\'e bordism are monoidal functors}
\label{Nov8}

With the notation of Example \ref{Nov12}, there are product maps
\[
\bM_{\pi,Z,w}\wedge
\bM_{\pi',Z',w'}
\to
\bM_{\pi\times\pi',Z\times Z',w\times w'}
\]
and
\[
\bM^R\wedge \bM^S
\to
\bM^{R\otimes S}
\]
induced by the operations $\times$ and $\otimes$ of Lemmas \ref{prod} and
\ref{prod2}.  There is also a unit map
\[
\bS\to \bM_{e,*,1}
\]
defined as follows.  The one-point space
gives a $*$-ad of degree 0, and this induces a map of spaces from
$S^0$ to the 0-th space of $\bM_{e,*,1}$; the unique extension of this to a map
of symmetric spectra is the desired unit map.  Similarly there is a unit map
\[
\bS\to\bM^\Z
\]
determined by the $*$-ad $(\Z,\Z,\varphi)$, where $\varphi$ is the identity
map.

Assumption \ref{n23} implies that the categories $\T$ and $\R$ introduced in 
Section \ref{fun} are strict monoidal categories.

\begin{definition}
Let $\Mgeom$ be the functor from $\T$ to the category of symmetric spectra 
which takes $(\pi,Z,w)$ to $\bM_{\pi,Z,w}$.  Let $\Msym$ be the functor from 
$\R$ to the category of symmetric spectra which takes $R$ to $\bM^R$.
\end{definition}

\begin{thm}
\label{Nov11}
$\Mgeom$ and $\Msym$ are monoidal functors.  That is, the
following diagrams involving $\Msym$ strictly commute, and similarly for 
$\Mgeom$.

\[
\xymatrix{
(\bM^R\wedge \bM^S)\wedge \bM^T
\ar[rr]^\cong
\ar[d]_\otimes
&
&
\bM^R\wedge(\bM^S\wedge \bM^T)
\ar[d]^\otimes
\\
\bM^{R\otimes S}\wedge \bM^T
\ar[dr]_\otimes
&
&
\bM^R\wedge \bM^{S\otimes T}
\ar[dl]^\otimes
\\
&
\bM^{R\otimes S\otimes T}
&
}
\]
\[
\xymatrix{
\bS\wedge \bM^R
\ar[r]^-\cong
\ar[d]
&
\bM^R
\\
\bM^\Z\wedge \bM^R
\ar[r]^-\otimes
&
\bM^{\Z\otimes R}
\ar[u]_-=
}
\qquad
\xymatrix{
\bM^R
\wedge 
\bS
\ar[r]^-\cong
\ar[d]
&
\bM^R
\\
\bM^R
\wedge 
\bM^\Z
\ar[r]^-\otimes
&
\bM^{R\otimes \Z}
\ar[u]_-=
}
\]
\end{thm}

\bigskip

The proof is similar to that of Theorem \ref{n24}.

\begin{remark}
One can define a ``module functor'' over a monoidal functor in the evident 
way.  Then the functor $\Mquad$ from $\R$ to the category of symmetric spectra
which takes $R$ to $\bM_R$ is a module functor over $\Msym$.
\end{remark}

\begin{remark}
The map
\[
\Sig:\bM_{\pi,Z,w}\to \bM^{\Z[\pi]^w}
\]
in Example \ref{Nov12} is not a monoidal natural transformation,
because the functor which takes a space to its singular chain
complex does not take Cartesian products to tensor products.  We will return to
this point in the sequel.
\end{remark}

\appendix
\section{Ball complex structures on PL manifolds and homology manifolds}

We begin with an elementary fact.

\begin{prop}
\label{f7}
Let $X$ be a compact oriented homology manifold of dimension $n$ with a 
regular CW complex structure such that $\partial X$ is a subcomplex.

{\rm (i)} Let $S$ be the set of $n$-dimensional cells, with their induced 
orientations.  Then the cellular chain $\sum_{\sigma\in S}\, \sigma$ 
represents the fundamental class $[X]\in H_n(X,\partial X)$.

{\rm (ii)} Let $T$ be the set of $(n-1)$-dimensional cells of $\partial X$, 
with their induced orientations.   Then
\[
\partial\left(
\sum_{\sigma\in S}\, \sigma
\right)
=
\sum_{\tau\in T}\, \tau
\]
\end{prop}

\begin{proof}
Part (i) follows from the fact that if $\sigma\in S$ and $x$ is in the interior
of $\sigma$ then the image of $[X]$ in $H_n(X,X-\{x\})$ is represented by
$\sigma$.  Applying part (i)  to $\partial X$ gives part (ii).
\end{proof}

Next we recall from
\cite{MR0400243}
that the concept of barycentric subdivision generalizes from simplicial 
complexes to ball complexes.

Let $K$ be a ball 
complex.  For each cell $\sigma$ of $K$, choose a point $\hat{\sigma}$ in the 
interior of $\sigma$ and a PL isomorphism from $\sigma$ to the cone 
$C(\partial\sigma)$ which takes $\hat{\sigma}$ to the cone point.  With this 
data, $K$ is a ``structured cone complex'' (\cite[page 274]{MR0400243}).  By 
\cite[Proposition 2.1]{MR0400243}, $K$ has a subdivision $\hat{K}$ which is 
a simplicial complex
with vertices $\hat{\sigma}$. A set of vertices in 
$\hat{K}$ spans a simplex in $\hat{K}$ if and only if it has the form
\[
\{\hat{\sigma}_1,\ldots,\hat{\sigma}_k\}
\]
with $\sigma_1\subset\cdots\subset\sigma_k$.  

Recall that if $S$ is a simplicial complex and $v$ is a vertex of $S$ then 
the {\it closed star} $\text{st}(v)$ is the subcomplex consisting of all 
simplices which contain $v$ together with all of their faces.  The {\it link}
$\text{lk}(v)$ is the subcomplex of $\text{st}(v)$ consisting of simplices
that do not contain $v$.  The realization $|\text{st}(v)|$ is the cone 
$C(|\text{lk}(v)|)$.  

\begin{prop}
\label{f8}
Let $K$ be a ball complex and let $\sigma$ be a cell of $K$.  

{\rm (i)}  For each cell $\tau$ with $\tau\supsetneq \sigma$ the 
subspace $|\lk(\hat\sigma)|\cap \tau$ is a PL ball, and these subspaces,
together with the cells of $K$ contained in $\partial\sigma$, are a 
ball complex structure on $|\lk(\hat\sigma)|$. 

{\rm (ii)}  For each cell $\tau$ with $\tau\supset \sigma$ the 
subspace $|\st(\hat\sigma)|\cap \tau$ is a PL ball, and these subspaces,
together with the cells of $|\lk(\hat\sigma)|$, 
are a ball complex structure on $|\st(\hat\sigma)|$.

\end{prop}

\begin{proof}
Let $\tau\supsetneq \sigma$.
$|\tau|$ inherits a ball complex structure from $K$, and 
$|\lk(\hat\sigma)|\cap \tau$ (resp., $|\st(\hat\sigma)|\cap \tau$) is the 
realization of the link (resp., star) of $\hat\sigma$ with respect to this 
structure.  It is a PL ball because $|\tau|$ is a PL manifold and
$\hat\sigma$ is a point of its boundary.
\end{proof}

Next we recall from \cite{MR0400243} that the concept of dual cell 
generalizes from simplicial complexes to ball complexes.  
For each cell $\sigma$ of 
$K$, let $D(\sigma)$ (resp., $\dot{D}(\sigma)$) be the subcomplex of $\hat{K}$ 
consisting of simplices
$\{\hat{\sigma}_1,\ldots,\hat{\sigma}_k\}$ with $\sigma\subset\sigma_i$ 
(resp., $\sigma\subsetneq\sigma_i$) for all $i$. 

Two simplices $s,s'$ of a simplicial complex $S$ are {\it joinable} if
their vertex sets are disjoint and the union of their vertices spans a simplex
of $S$; this simplex is called the {\it join}, denoted $s*s'$.  Two
subcomplexes $A$ and $B$ of $S$ are joinable if each pair $s\in A$, $s'\in B$
is joinable, and the join $A*B$ is the subcomplex consisting of the simplices 
$s*s'$ and all of their faces.

\begin{lemma}
\label{f1}
If $K$ is a ball complex and $\sigma$ is a cell of $K$ then
\[
\mathrm{lk}(\hat{\sigma},\hat{K})=\partial \sigma * \dot{D}(\sigma).
\]
\end{lemma}

The proof is immediate from the definitions.  

\begin{prop}
\label{f5}
Let $(L,L_0)$ be a ball complex pair such that $|L|$ is a PL manifold of
dimension $n$ with boundary $|L_0|$.  Let $\sigma$ be a cell of $L$ of
dimension $m$.

{\rm (i)}
If $\sigma$ is not a cell of $L_0$ then $|D(\sigma)|$ is a PL $(n-m)$-ball
with boundary $|\dot{D}(\sigma)|$ and with $\hat{\sigma}$ in 
its interior.

{\rm (ii)}
If $\sigma$ is a cell of $L_0$ then $|\dot{D}(\sigma)|$ is a PL $(n-m-1)$-ball
and $|D(\sigma)|$ is a PL $(n-m)$-ball with $|\dot{D}(\sigma)|$ and 
$\hat{\sigma}$ on its boundary.
\end{prop}

\begin{proof}
For part (i), $|\mathrm{lk}(\hat{\sigma})|$ is a  PL $(n-1)$-sphere.
Lemma \ref{f1} and Theorem 1 of 
\cite{MR0261587} imply that $|\dot{D}(\sigma)|$ is a PL $(n-m-1)$-sphere, and
hence $|D(\sigma)|$ is a PL $(n-m)$ ball. 

Part (ii) is similar. 
\end{proof}

\section{Comparison of the Thom and Quinn models for $M\Top$}
\label{a1}

The method of Section \ref{classical} gives an ad theory $\ad_\Top$.  
Let $Q_\Top$ denote the Quinn spectrum obtained from this  ad theory. 
In this appendix we prove

\begin{prop}
\label{a2}
The Thom spectrum $M\Top$ is weakly equivalent to $\bQ_\Top$. 
\end{prop}

\begin{remark}
The analogous result for $\STop$ is true and the same proof works with minor
changes.
\end{remark}

We use the definition of weak equivalence in \cite{MR1806878}, which is
the same as the classical definition: the $i$-th homotopy group of a spectrum
$\mathbf X$ is $\colim_k \pi_{i+k} X_k$, and a weak equivalence is a map which
induces an isomorphism of homotopy groups.  The proof will give an explicit
chain of weak equivalences between $M\Top$ and $\bQ_\Top$.

First we need some background.  For a space $X$ let $CX$ be the unreduced cone
$I\wedge X_+$ (with $1$ as the basepoint of $I$) and let $S_\bu X$ denote the
usual singular complex of $X$,
considered as a semisimplicial set.
There is a homeomorphism
\[
\iota:C\Delta^n\to \Delta^{n+1}
\]
which takes $(t,(s_0,\ldots,s_n))$ to $((1-t)s_0,\ldots,(1-t)s_n,t)$.
There is a map 
\[
\kappa: \Sigma S_\bu X\to S_\bu (\Sigma X)
\]
(where the first $\Sigma$ is the Kan suspension defined in Section
\ref{spectrum}) which takes $f:\Delta^n\to X$ to the composite
\[
\Delta^{n+1}\xrightarrow{\iota^{-1}}
C\Delta^n\xrightarrow{Cf}
CX\to \Sigma X.
\]
It's straightforward to check that the diagram
\begin{equation}
\label{a3}
\xymatrix
{
\Sigma |S_\bu X|
\ar[d]
\ar[r]^-\lambda
&
|\Sigma S_\bu X|
\ar[d]_{|\kappa|}
\\
\Sigma X
&
|S_\bu \Sigma X|
\ar[l]
}
\end{equation}
commutes, where $\lambda$ was given in the proof of Lemma 
\ref{n17}.

Now let $\bX$ be the spectrum whose $k$-th space is $|S_\bu T(\Top_k)|$, with 
structure maps
\[
\Sigma|S_\bu T(\Top_k)|
\xrightarrow{\lambda}
|\Sigma S_\bu T(\Top_k)|
\xrightarrow{|\kappa|}
|S_\bu \Sigma T(\Top_k)|
\to
|S_\bu T(\Top_{k+1})|,
\]
where 
the third map is induced by the structure map of $M\Top$.
The commutativity of diagram \eqref{a3} shows that the natural maps
\[
|S_\bu T(\Top_k)| \to T(\Top_k)
\]
give a map of spectra $\bX\to M\Top$, and this map is a weak equivalence by
\cite[Proposition 2.1]{MR0300281}.

Next let 
$S^{\mathord{\pitchfork}}_\bu T(\Top_k)$ be the sub-semisimplicial set of 
$S_\bu T(\Top_k)$ consisting of maps whose restrictions to each face of
$\Delta^n$ are transverse to the zero section.
Let $\bY$ be the subspectrum of $\bX$ with $k$-th space 
$|S^{\mathord{\pitchfork}}_\bu T(\Top_k)|$. Then the inclusion
$\bY\hookrightarrow \bX$ is a weak equivalence by \cite[Section
9.6]{MR1201584}.

Finally, let $P_k$ be the sequence of semisimplical sets associated to the ad 
theory $\ad_\Top$ as in Section \ref{spectrum}.  For each $k$ let
\[
\mu_k:S^{\mathord{\pitchfork}}_\bu T(\Top_k)
\to 
P_k
\]
be the map which take a transverse map to the $\Delta^n$-ad it determines.
It is straightforward to check that the $|\mu_k|$ give a map of spectra
\[
\mathfrak{m}: \bY\to \bQ_\Top.  
\]
We need to show that $\mathfrak{m}$ is a weak equivalence.

Let $\Omega_*(\Top)$ denote the topological bordism groups.
It is a folk theorem that there is an isomorphism
\[
\Omega_*(\Top)\to \pi_*(M\Top)
\]
whose construction is similar to that on pages 19--20 of \cite{Stong}.  The
proof on pages 19--23 of \cite{Stong} appears to go through with suitable
changes, using some combination of 
\cite{MR0161346}, 
\cite[Theorem 1]{MR0198490}, 
\cite{Marin},
\cite[Proposition IV.8.1]{MR0645390}, 
\cite{Kister}, 
and \cite[Section 9.6]{MR1201584}, 
but the details do not seem to have been written down in the
literature.  In what follows we assume that this folk theorem is true.

Since $S^{\mathord{\pitchfork}}_\bu T(\Top_k)$ is a Kan complex, we can define
a map 
\[
\nu_k: \pi_*|S^{\mathord{\pitchfork}}_\bu T(\Top_k)|\to \Omega_{*-k}(\Top)
\]
as follows: an element of $\pi_n|S^{\mathord{\pitchfork}}_\bu T(\Top_k)|$ is
represented by a map $f:\Delta^n\to T(\Top_k)$ which takes all faces to the basepoint and is
transverse to the 0 section; then $f^{-1}$ of the 0 section is a manifold and
we let $\nu_k([f])$ be the bordism class of this manifold.  The $\nu_k$ induce
a map
\[
\nu:\pi_*\bY\to \Omega_*(\Top)
\]
which is an isomorphism because the composite
\[
\Omega_*(\Top)
\xrightarrow{\cong}
\pi_*(M\Top)
\xleftarrow{\cong}
\pi_*Y
\xrightarrow{\nu}
\Omega_*(\Top)
\]
is the identity.  The diagram
\[
\xymatrix{
\pi_*Y
\ar[r]^-\nu
\ar[rd]_{\mathfrak{m}_*}
&
\Omega_*(\Top)
\ar[d]_=
\\
&
\pi_*\bQ_\Top
}
\]
commutes, and it follows that $\mathfrak{m}$ is a weak equivalence.
This completes the proof of Proposition \ref{a2}.

\bibliographystyle{amsalpha}

\bibliography{bordmult}

\end{document}